\documentclass[11pt]{amsart}

\usepackage{amsmath}
\usepackage{amssymb}

\newcommand{\C}{{\mathbb C}}       
\newcommand{\R}{{\mathbb R}} 
\newcommand{\Rn}{{\mathbb R}^n} 
\newcommand{\Rd}{{\mathbb R}^d} 
\newcommand{\N}{{\mathbb N}}       %
\newcommand{\Z}{{\mathbb Z}}       
\newcommand{\DD}{{\mathcal D}}
\newcommand{\FF}{{\mathcal F}}
\newcommand{\HH}{{\mathcal H}}
\newcommand{\GG}{{\mathcal G}}

\newcommand{\UU}{{\mathcal U}}

\newcommand{\AZ}{{\mathcal A}}

\newcommand{\ha}{{\mathcal H}}
\newcommand{\RR}{{\mathcal R}}

\newcommand{\cF}{{\mathcal F}}
\newcommand{\cB}{{\mathcal B}}
\newcommand{\cG}{{\mathcal G}}
\newcommand{\cT}{{\mathcal T}}
\newcommand{\ra}{\rightarrow}

\newcommand{\wfi}{{\wt \vphi_R}}
\newcommand{\wf}{{\wt \vphi}}
\renewcommand{\emptyset}{{\varnothing}}

\newcommand{\mang}{{\measuredangle}}

\newcommand{\diam}{{\rm diam}}
\newcommand{\dist}{{\rm dist}}
\newcommand{\sgn}{{\rm sgn}}

\newcommand{\rf}[1]{{\eqref{#1}}}

\newcommand{\supp}{\operatorname{supp}}

\newcommand{\vphi}{{\varphi}}
\newcommand{\ve}{{\varepsilon}}
\newcommand{\vv}{}{{\vspace{2mm}}}
\newcommand{\vvv}{}{{\vspace{3mm}}}
\newcommand{\wt}[1]{{\widetilde{#1}}}

\newcommand{\de}{{\partial_{out}}}

\renewcommand{\de}{{\delta}}
\renewcommand{\a}{{\alpha}}
\newcommand{\xh}{x^H}
\newcommand{\stm}{{\setminus}}
\newcommand{\cys}{{\wt\chi_{y,s}}}

\newcommand{\rest}{{\lfloor}}

\newtheorem{theorem}{Theorem}[section]
\newtheorem{lemma}[theorem]{Lemma}

\newtheorem{coro}[theorem]{Corollary}
\newtheorem{propo}[theorem]{Proposition}

\newtheorem*{lemma*}{Lemma}

\theoremstyle{definition}
\newtheorem{definition}[theorem]{Definition}

\theoremstyle{remark}
\newtheorem{rem}[theorem]{Remark}

\numberwithin{equation}{section}

\newcommand{\brem}{\begin{rem}}
\newcommand{\erem}{\end{rem}}

\title{Square functions and uniform rectifiability}

\author{Vasileios Chousionis}
\address{Vasileios Chousionis. Department of Mathematics and Statistics,
P.O. Box 68,  FI-00014 University of Helsinki, Finland}
\email{vasileios.chousionis@helsinki.fi}
\thanks{V.C. was funded by the Academy of Finland Grant SA 267047. Also, partially supported by the ERC Advanced Grant 320501, while visiting Universitat Aut\`onoma de Bar\-ce\-lo\-na}
\author{John Garnett}
\address{John Garnett. Department of Mathematics, University of California at Los Angeles. 6363 Math Sciences Building, Los Angeles, CA 90095-1555}
\email{jbg@math.ucla.edu} 
\thanks{J.G. was partially supported by NSF DMS 1217239 and the IPAM  long program Interactions Between Analysis and Geometry, Spring 2013.}
\author{Triet Le}
\address{Triet Le. Department of Mathematics, University of Pennsylvania. David Rittenhouse Lab. 209 South 33rd Street, Philadelphia, PA 19104.}
\email{trietle@math.upenn.edu}
\thanks{T.L. was partially supported by NSF DMS 1053675 and 
the IPAM long program  Interactions Between Analysis and Geometry, Spring 2013.}
\author{Xavier Tolsa}
\address{Xavier Tolsa. Instituci\'{o} Catalana de Recerca i Estudis Avan\c{c}ats (ICREA) and Departament de Ma\-te\-m\`a\-ti\-ques, Universitat Aut\`onoma de Bar\-ce\-lo\-na, Catalonia}
\email{xtolsa@mat.uab.cat}
\thanks{X.T. was funded by the an Advanced Grant of the European Research
Council (programme FP7/2007-2013), by agreement 320501. Also, partially supported by grants 2009SGR-000420 (Generalitat de Catalunya) and MTM-2010-16232 (Spain).}

\begin{document}

\begin{abstract}
In this paper it is shown that an Ahlfors-David $n$-dimensional measure $\mu$
on $\R^d$ is uniformly $n$-rectifiable if and only if for any ball $B(x_0,R)$ centered at $\supp(\mu)$,
$$
\int_0^R \int_{x\in B(x_0,R)} \left|\frac{\mu(B(x,r))}{r^n} - \frac{\mu(B(x,2r))}{(2r)^n}
\right|^2\,d\mu(x)\,\frac{dr}r \leq c\, R^n.$$
Other characterizations of uniform $n$-rectifiability in terms of
smoother square functions are also obtained.
\end{abstract}

\maketitle

\section{Introduction}

Given $0<n<d$, a Borel set $E \subset \Rd$ is said to be  $n$-rectifiable if it is contained
in a countable union  of $n$-dimensional $C^1$ manifolds and  
a set of zero $n$-dimensional Hausdorff measure $\ha^n$.
On the other hand, a Borel measure $\mu$ in $\R^d$ is called $n$-rectifiable if it is of the form
$\mu = g\,\HH^n|_E$, where $E$ is a Borel $n$-rectifiable set and $g$ is positive and $\HH^n$ integrable on $E$.
Rectifiability is a qualitative notion,  but
 David and Semmes in their landmark works \cite{DS1} and \cite{DS2} introduced the more quantitative notion of uniform rectifiability. 
To define uniform rectifiability  we need first to recall the notion of Ahlfors-David regularity.

We say a Radon measure $\mu$ in $\R^d$
is $n$-dimensional Ahlfors-David regular with constant $c_0$ if
\begin{equation}
\label{ad}
c_0^{-1}r^{n}\leq \mu(B(x,r)) \leq c_0\,r^{n}\quad\mbox{for all $x\in\supp(\mu)$,
$0<r\leq\diam(\supp(\mu))$}.
\end{equation}
 For short, we sometimes omit the constant $c_0$ and call $\mu$  $n$-AD-regular. It follows easily that such a measure $\mu$ must be of the  form  $\mu =h\,\HH^n|_{\supp(\mu)}$, where  $h$ is a positive function bounded from above and from below.
 
An $n$-AD-regular measure $\mu$ is  uniformly  $n$-rectifiable if there exist $\theta, M >0$ such that for all $x \in \supp(\mu)$ and all $r>0$ 
there exists a Lipschitz mapping $\rho$ from the ball $B_n(0,r)$ in $\R^{n}$ to $\R^d$ with $\text{Lip}(\rho) \leq M$ such that$$
\mu (B(x,r)\cap \rho(B_{n}(0,r)))\geq \theta r^{n}.$$
When $n =1$, $\mu$ is uniformly $1$-rectifiable if and only if $\supp(\mu)$ is contained in a rectifiable curve in $\R^d$ on which the arc length measure satisfies  \eqref{ad}. A Borel set $E \subset \Rd $ is
 $n$-AD-regular if $\mu=\ha^n|_{E}$ is $n$-AD-regular, and it is called uniformly $n$-rectifiable if, further, $\ha^n|_{E}$ is uniformly $n$-rectifiable.
Thus $\mu$ is an uniformly $n$-rectifiable measure if and only if $\mu = h\,\HH^n|_E$ where $h >0$ is bounded
above and below and $E$ is an uniformly $n$-rectifiable closed set. 

Uniform rectifiability is closely connected to the geometric study of singular integrals. In \cite{D2} David proved that if $E \subset \Rd$ is uniformly $n$-rectifiable, then for any 
convolution kernel $K:\mathbb{R}^{d} \setminus \{0\}\rightarrow \mathbb{R}$
satisfying
\begin{equation}
\label{ker}
K(-x)=-K(x)\;  \text{ and }\ \;\left| \nabla^j K(x)\right| \leq c_j \left| x\right| ^{-n-j},  \text{ for}\ x\in \mathbb{R}^{d}\setminus \{0\}, \ j=0,1,2,\dots,
\end{equation} 
the associated singular integral operator $T_Kf(x) = \int K(x-y)\,f(y)\,d\HH^n|_E(y)$  
is bounded in $L^2(\ha^n|_E)$. David and Semmes in \cite{DS1} proved conversely that the
$L^2(\ha^n|_E)$-boundedness of all singular integrals $T_K$ with kernels satisfying \eqref{ker} implies
that  $E$ is uniformly
$n$-rectifiable. However if one only assumes the boundedness of some particular singular integral operators satisfying
\eqref{ker}, then the situation becomes much more delicate.

In \cite{MMV} Mattila, Melnikov and Verdera proved that if $E$ is an $1$-AD regular set, the Cauchy transform is bounded in $L^2(\ha^n|_E)$ if and only if $E$ is uniformly $1$-rectifiable. It is remarkable that their proof depends crucially on a special subtle positivity property of the Cauchy kernel related to the so-called Menger curvature. 
See \cite{CMPT} for other examples of $1$-dimensional homogeneous convolution kernels 
whose $L^2$-boundedness is equivalent to uniform rectifiability, again because of Menger curvature. 
Recently in \cite{ntov} it was shown that in the codimension $1$ case, that is, for $n=d-1$, 
if  $E$ is $n$-AD-regular, then the vector valued Riesz kernel $x/|x|^{n+1}$ defines a bounded operator on $L^2(\ha^n|_E)$ if and only if $E$ is uniformly $n$-rectifiable. In this case, the notion of Menger curvature is not 
applicable and the proof relies instead on the harmonicity of the kernel $x/|x|^{n+1}$.
It is an open problem if the analogous result holds for 
$1<n<d-1$. 

In this paper we prove several characterizations of uniform $n$-rectifiability in terms of square functions. Our first characterization involves the following difference of densities 
$$\Delta_\mu(x,r) := \frac{\mu(B(x,r))}{r^n} - \frac{\mu(B(x,2r))}{(2r)^n}$$
and reads as follows.

\begin{theorem}
\label{main1}
Let $\mu$ be an $n$-AD-regular measure. Then $\mu$ is uniformly $n$-rectifiable if and only if there exists a constant $c$ such that, for any ball $B(x_0,R)$ centered at $\supp(\mu)$,
\begin{equation}\label{eqsq2}
\int_0^R \int_{x\in B(x_0,R)} |\Delta_\mu(x,r)|^2\,d\mu(x)\,\frac{dr}r \leq c\, R^n.
\end{equation}
\end{theorem}

Recall that a celebrated theorem of Preiss \cite{Preiss} asserts that a Borel measure $\mu$ in $\R^d$ is $n$-rectifiable if and only if the density
$\lim_{r\to0}\dfrac{\mu(B(x,r))}{r^n}$
exists and is positive for $\mu$-a.e.\ $x\in\R^d$. In a sense, Theorem \ref{main1} can be considered as a square function version of
Preiss' theorem for uniform rectifiability. On the other hand, let us mention that the ``if'' implication
in our theorem relies on some of the deep results by Preiss in \cite{Preiss}.

It is also worth comparing Theorem \ref{main1} to some earlier results
from Kenig and Toro \cite{KT},
David, Kenig and Toro \cite{DKT} and Preiss, Tolsa and Toro \cite{PTT}.
In these works it is shown among other things that, given $\alpha>0$, there
exists $\beta(\alpha)>0$ such that 
if $\mu$ is $n$-AD-regular and for each compact set $K$ there exists some constant $c_K$
such that
$$\left|\frac{\mu(B(x,r))}{r^n} - \frac{\mu(B(x,tr))}{(tr)^n}\right| \leq c_K \,r^\alpha
\quad \mbox{for $1<t\leq 2$, $x\in K\cap\supp(\mu)$, $0<r\leq 1$,}$$
then $\mu$ is supported on an $C^{1+\beta}$ $n$-dimensional manifold union  a closed
set with zero $\mu$-measure. This result can be thought of as the H\"older version
of one of the implications in Theorem \ref{main1}.

We also want to mention  the forthcoming work \cite{ADT} by Azzam, David and Toro for some other conditions on a doubling measure which imply rectifiability. One of the conditions in \cite{ADT} quantifies the difference of the measure at different close scales in terms of  the Wasserstein distance $W_1$. 
In our case, the square function in Theorem \ref{main1} just involves the difference of the $n$-dimensional densities of two concentric balls such that the largest radius doubles the smallest one.

Motivated by the recent work \cite{LM} studying local scales on curves and surfaces, which was
the starting point of this paper's research,  we also prove smooth versions of  Theorem \ref{main1}. For any Borel function $\vphi: \Rd \ra \R$ let 
$$\vphi_t(x)=\frac{1}{t^n} \vphi \left(\frac xt \right), \, t>0$$
and define
$$\Delta_{\mu,\vphi} (x,t):= \int\bigr(\vphi_t (y-x)-\vphi_{2t}(y-x)\bigr)\,d\mu(y),$$
whenever the integral makes sense.
If $\vphi$ is smooth, let
$$\partial_\vphi (x,t)=t \partial_t\, \vphi_t(x)$$
and define
$$\wt \Delta_{\mu,\vphi} (x,t):= \int\partial_\vphi (y-x,t)\,d\mu(y),$$
again whenever the integral makes sense. Our second theorem characterizes  uniform $n$-rectifiable $n$-AD-regular measures using the square functions associated with   $\Delta_{\mu,\vphi}$ and $\wt \Delta_{\mu,\vphi}$. 

\begin{theorem} 
\label{main2}Let $\vphi:\Rd \ra \R$ be of the form $e^{-|x|^{2N}}$, with $N\in\N$, or $(1+|x|^2)^{-a}$, with $a>n/2$. 
Let $\mu$ be an $n$-AD-regular measure in $\R^d$. The following are equivalent:
\begin{enumerate}
\item[(a)] $\mu$ is uniformly $n$-rectifiable.
\item [(b)]There exists a constant $c$ such that for any ball $B(x_0,R)$ centered at $\supp(\mu)$,
\begin{equation}
\label{eqsqdif}
\int_0^R \int_{x\in B(x_0,R)} |\Delta_{\mu,\vphi}(x,r)|^2\,d\mu(x)\,\frac{dr}r \leq c\, R^n.
\end{equation}
\item[(c)] There exists a constant $c$ such that for any ball $B(x_0,R)$ centered at $\supp(\mu)$,
\begin{equation}
\label{eqsqdifsm}
\int_0^R \int_{x\in B(x_0,R)} |\wt \Delta_{\mu,\vphi}(x,r)|^2\,d\mu(x)\,\frac{dr}r \leq c\, R^n.
\end{equation}
\end{enumerate}
\end{theorem}

The functions $\varphi_t$  above are radially symmetric and (constant multiples of) approximate
identities on any $n$-plane containing the origin. The definitions of $\Delta_{\mu,\varphi}(x,t)$ and $\wt \Delta_{\mu,\vphi}(x,t)$ arise from convolving the measure $\mu$ with the kernels $\varphi_t(x) - \varphi_{2t}(x)$ and $\partial_{\varphi}(x,t)$, respectively.
 Note that $\varphi_t(x) - \varphi_{2t}(x)$ is a discrete approximation to $\partial_{\varphi}(x,t)$. 
Note also that  the quantities $\Delta_{\mu}(x,t),\,\Delta_{\mu,\varphi}(x,t)$ and $\wt \Delta_{\mu,\vphi}(x,r)$ are identically zero whenever $\mu= \ha^n|_{L}$,  $L$ is an $n$-plane, and $x\in L$.  

For each integer $k>0$, let
\begin{equation*}
\wt\Delta_{\mu,\vphi}^k(x,t) = \int \partial_\vphi^k(y-x,t)\ d\mu(y), \mbox{ where }\partial_\vphi^k(x,t) = t^k\partial_t^k\vphi_t(x).
\end{equation*}
Similarly, let
\begin{equation*}
\Delta_{\mu,\vphi}^k(x,t) = \int D^k\left[\vphi_t\right](y-x)\ d\mu(y),\\
\end{equation*}
where
$$
D^k [\vphi_t](x) = D^{k-1}[D \vphi_t](x),\;\mbox{ and } \;D \vphi_t(x) =\vphi_t(x) - \vphi_{2t}(x).
$$

By arguments analogous to the ones of Theorem \ref{main2}, we obtain the following equivalent square function conditions for uniform rectifiability.
\begin{propo}
\label{prop1}Let $\vphi:\Rd \ra \R$ be of the form $e^{-|x|^{2N}}$, with $N\in\N$, or $(1+|x|^2)^{-a}$, with $a>n/2$. 
Let $\mu$ be an $n$-AD-regular measure in $\R^d$ and $k>0$. The following are equivalent:
\begin{enumerate}
\item[(a)] $\mu$ is uniformly $n$-rectifiable.
\item [(b)]There exists a constant $c_k$ such that for any ball $B(x_0,R)$ centered at $\supp\mu$,
\begin{equation}
\label{prop_eq2}
\int_0^R \int_{x\in B(x_0,R)} |\Delta_{\mu,\vphi}^k(x,r)|^2\,d\mu(x)\,\frac{dr}r \leq c_k\, R^n.
\end{equation}
\item[(c)] There exists a constant $c_k$ such that for any ball $B(x_0,R)$ centered at $\supp\mu$,
\begin{equation}
\label{prop_eq3}
\int_0^R \int_{x\in B(x_0,R)} |\wt \Delta_{\mu,\vphi}^k(x,r)|^2\,d\mu(x)\,\frac{dr}r \leq c_k\, R^n.
\end{equation}
\end{enumerate}
\end{propo}

Proposition \ref{prop1} is  in the same spirit as the characterization of Lipschitz function spaces in Chapter V, Section 4 of \cite{Stein}. 

There are other characterizations of uniform $n$-rectifiability via square functions in the literature. Among the most relevant of these is a condition in terms of the $\beta$-numbers of Peter Jones.
For $x \in \supp(\mu)$ and $r>0,$ consider the coefficient
$$\beta^\mu_1(x,r)= \inf_L \int_{B(x,r)} \frac{\dist(y,L)}{r^{n+1}}d \mu (y),$$
where the infimum is taken over all $n$-planes $L$. 
Like $\Delta_\mu(x,r)$, $\beta^\mu_1(x,r)$  is a dimensional coefficient, but while $\beta^\mu_1(x,r)$ measures how close 
$\supp(\mu)$ is to some $n$-plane, $\Delta_\mu(x,r)$ measures
the oscillations of $\mu$. 
In \cite{DS1},
David and Semmes  proved that   $\mu$ is uniformly $n$-rectifiable if and only if $\beta^\mu_1(x,r)^2 d x \frac{dr}{r} $ is a Carleson measure on $\supp(\mu) \times (0, \infty)$, that is, \eqref{eqsq2} is satisfied with $\Delta_\mu(x,r)$ replaced by $\beta_1^\mu (x,r)$.  

The paper is organized as follows. In Section \ref{sec:prel} we provide the preliminaries for the proofs of Theorems \ref{main1} and \ref{main2}. In Section \ref{sec:bddur} we show first that the boundedness of the smooth square functions in \eqref{eqsqdif} and \eqref{eqsqdifsm} implies uniform rectifiability. Using suitable convex combinations,  we then show that \eqref{eqsq2} implies \eqref{eqsqdif}, and thus 
establish one of the implications in Theorem \ref{main1}. In Section \ref{sec:urbddsm} we prove that uniform $n$-rectifiability implies \eqref{eqsqdif} and \eqref{eqsqdifsm}, and thereby 
complete the proof of Theorem \ref{main2}. In Section \ref{sec:urbdd} we prove that \eqref{eqsq2} holds if $\mu$ is uniform $n$-rectifiable; this is the most delicate part of the paper because of  complications which arise from the non-smoothness of the function $r^{-n}\chi_{B(0,r)} - (2r)^{-n}\chi_{B(0,2r)}$. Finally, in Section \ref{sec_prop} we outline the proof for Proposition \ref{prop1}.

Throughout the paper the letter $C$ stands
for some constant which may change its value at different
occurrences. The notation $A\lesssim B$ means that
there is some fixed constant $C$ such that $A\leq CB$,
with $C$ as above. Also, $A\approx B$ is equivalent to $A\lesssim B\lesssim A$.  

\section{Preliminaries}\label{sec:prel}

\subsection{The David cubes}

Below we will need to use the David lattice $\DD$ of ``cubes'' associated with $\mu$ (see \cite[Appendix 1]{Da}, for example). 
Suppose for simplicity that $\mu(\R^d)=\infty$.
In this case,
$\DD=\bigcup_{j\in\Z}\DD_j$ and each set $Q\in\DD_j$, which is called a cube, satisfies 
$\mu(Q)\approx 2^{-jn}$ and $\diam(Q)\approx2^{-j}$. In fact, we will assume that
$$c^{-1}2^{-j}\leq \diam(Q)\leq 2^{-j}.$$
We set $\ell(Q):=2^{-j}$. 
For $R \in \DD$, we denote by $\DD(R)$ the family of all cubes $Q \in \DD$ which are contained in $R$. In the case when $\mu(\R^d)<\infty$ and $\diam(\supp(\mu))\approx2^{-j_0}$, then $\DD=\bigcup_{j\geq j_0}
\DD_j$. The other properties of the lattice $\DD$ are the same as in the previous case.

\subsection{The $\alpha$ coefficients}

The so called $\alpha$ coefficients from \cite{Tolsa-lms} play a crucial role in our proofs. They are defined as follows.
Given a closed ball $B\subset\R^d$ which intersects $\supp(\mu)$,
 and two finite Borel measures $\sigma$ and $\nu$
in $\R^d$ , we set
$$\dist_B(\sigma,\nu):= \sup\Bigl\{ \Bigl|{\textstyle \int f\,d\sigma  -
\int f\,d\nu}\Bigr|:\,{\rm Lip}(f) \leq1,\,\supp f\subset
B\Bigr\},$$
where ${\rm Lip}(f)$ stands for the Lipschitz constant of $f$.
It is easy to check that this is indeed a distance in the space of finite Borel measures supported in the interior of 
$B$. See [Chapter 14, Ma] for other properties of this distance.
Given a subset $\AZ$ of Borel measures, we set
$$\dist_B(\mu,\AZ) := \inf_{\sigma\in\AZ}\dist_B(\mu,\,\sigma).$$
We define
$$
\alpha_\mu^n(B) := \frac1{r(B)^{n+1}}\,\inf_{c\geq0,L} \,\dist_{B}(\mu,\,c\HH^n_{|L}),$$
where $r(B)$ stands for the radius of $B$ and
the infimum is taken over all the constants $c\geq0$ and all the $n$-planes $L$. To simplify notation, 
we will write $\alpha(B)$ instead of $\alpha_\mu^n(B)$.

Given a cube $Q\in\DD$, let $B_Q$ be a ball with radius $10\ell(Q)$ with 
the same center as $Q$. We denote
$$\alpha(Q):=\alpha(B_Q).$$
We also denote by $c_Q$ and $L_Q$ a constant and an $n$-plane minimizing $\alpha(Q)$. We assume that $L_Q \cap \frac12 B_Q \neq \emptyset$.

The following is shown in \cite{Tolsa-lms}.

\begin{theorem}\label{teotol}
Let $\mu$ be an $n$-AD-regular measure in $\R^d$.
If $\mu$ is uniformly $n$-rectifiable, then there exists a constant $c$ such that
\begin{equation}\label{eqsq111}
\sum_{Q\subset R}\alpha(Q)^2\,\mu(Q)\leq c\,\mu(R)\qquad\mbox{for all $R\in\DD$.}
\end{equation}
\end{theorem} 

\subsection{The weak constant density condition}

Given $\mu$ satisfying \eqref{ad}, we  denote by $G(C,\ve)$ the subset of those $(x,r) \in \supp (\mu) \times (0, \infty)$ for which there exists a Borel measure $\sigma=\sigma_{x,r}$ satisfying 
\begin{enumerate}
\item $\supp (\sigma)=\supp(\mu)$,
\medskip
\item the $AD$-regularity condition \eqref{ad} with constant $C$,
\medskip
\item $|\sigma(B(y,t))-t^n|\leq \ve r^n$ for all $y\in \supp(\mu) \cap B(x,r)$ and all $0<t<r.$
\end{enumerate}

\begin{definition}
\label{wcd}
A Borel measure $\mu$ satisfies the \textit{weak constant density condition} (WCD) if there exists a positive constant $C$ such that the set 
$$G(C, \ve)^c:=[\supp(\mu) \times (0, \infty)]\setminus G(C,\ve)$$ is a Carleson set for every $\ve>0$, that is, for every $\ve>0$ there exists a constant $C(\ve)$ such that
\begin{equation}
\label{wcdcon}
\int_0^R \int_{B(x,R)} \chi_{G(C, \ve)^c}(x,r)\, d\mu(x) \frac{dr}r \leq C(\ve)R^n
\end{equation}
for all $x\in \supp(\mu)$ and $R>0$.
\end{definition}

\begin{theorem}
\label{wcdur}
Let $n\in (0,d)$ be an integer. An $n$-AD-regular measure $\mu$ in $\Rd$ is uniformly $n$-rectifiable if and only if it satisfies the weak constant density condition.
\end{theorem}

David and Semmes in \cite[Chapter 6]{DS1} showed that if $\mu$ is uniformly $n$-rectifiable, then it satisfies the WCD. In \cite[Chapter III.5]{DS2}, they also proved the converse in the cases when $n=1,2,d-1$. The proof of the converse for all codimensions was obtained very recently in \cite{toluni}. The arguments rely on two essential and deep ingredients: the so called bilateral weak geometric lemma of David and Semmes \cite{DS2}, and the (partial) 
  characterization of 
uniform measures by Preiss \cite{Preiss}.


\section{Boundedness of square functions implies uniform rectifiability} \label{sec:bddur}


In this section we assume that either  $\vphi(x)=e^{-|x|^{2N}}$, with $N\in\N$, or $\vphi(x)=(1+|x|^2)^{-a}$, with $a>n/2$, as in Theorem \ref{main2}. We will show that
 if \eqref{eqsqdif} or \eqref{eqsqdifsm} holds, then $\mu$ is uniformly $n$-rectifiable.  

We denote by $\wt\UU(\vphi,c_0)$  the class of $n$-AD-regular measures with constant $c_0$ such that
$f(r,x) = \vphi_r*\mu(x)$ is constant on $(0,\infty)\times \supp(\mu)$.


\begin{lemma}\label{lemcompact}
Let $\mu$ be an $n$-AD-regular measure such that $0\in\supp(\mu)$. For all
$\ve>0$ there exists  $\delta>0$ such that
if 
$$\int_{\delta}^{\delta^{-1}}\!\!
 \int_{x\in \bar B(0,\delta^{-1})} |\wt \Delta_{\mu,\vphi} (x,r)|\,d\mu(x)\,dr \leq \delta,$$
then
$$\dist_{B(0,1)}(\mu,\wt\UU(\vphi,c_0)) <\ve.$$
\end{lemma}

\begin{proof}
Suppose that there exists an $\ve>0$, and for each $m\geq 1$ there exists an $n$-AD-regular
 measure $\mu_m$ such that $0\in\supp(\mu_m)$,
\begin{equation}\label{eqass32}
\int_{1/m}^{m}
 \int_{x\in \bar B(0,m)} |\wt \Delta_{\mu_m,\vphi} (x,r)|\,d\mu_m(x)\,dr \leq \frac1m,
\end{equation}
and
\begin{equation}\label{equu12}
\dist_{B(0,1)}(\mu_m,\wt\UU(\vphi,c_0)) \geq\ve.
\end{equation}

By \eqref{ad}, we can replace $\{\mu_m\}$ by a subsequence converging  weak * (i.e. when tested against 
compactly supported continuous functions) to a measure $\mu$ and it is easy to check that $0 \in \supp(\mu)$ and that $\mu$ is also $n$-dimensional AD-regular with constant $c_0$. We claim that 
\begin{equation}\label{eqlimit}
\int_{0}^\infty\!
 \int_{x\in\R^d} |\wt \Delta_{\mu,\vphi} (x,r)|\,d\mu(x)\,dr=0.
 \end{equation}\

The proof of (\ref{eqlimit}) is elementary.  Fix  $m_0$ and 
 let $\eta > 0$. Because of \eqref{ad} and the decay conditions assumed for $\vphi$ there exists  $A > 2m_0$ so that
\begin{equation}\label{tight}
\sup_{1/m_0 \leq t \leq m_0} \int_{\bar B(0,2m_0)} \int_{|x - y| > A} |\partial_{\vphi}(x-y,t)| d\nu(y) d\nu(x)  < \frac{\eta}{m_0}
\end{equation}
whenever $\nu$ satisfies \eqref{ad}  with constant $c_0.$ Set  $K = [1/m_0,\,m_0] \times \bar B(0,2m_0)$ and let $\widetilde{\chi}$ be a continuous function with compact support such that $\chi_{B(0,A)}  \leq \widetilde{\chi} \leq 1.$   
Then, writing $\psi_t(x) = \partial_{\vphi}(x,t)$ we have by \eqref{tight}
$$
\iint_K |((1 -{\widetilde{\chi}})\psi_t) *\mu(x)| d\mu(x) dt < \eta,
$$
\noindent and by \eqref{eqass32}
$$
\iint_K |({\widetilde{\chi}}\psi_t) *\mu_m(x)| d\mu_m(x) dt < \eta + {{1} \over {m}}.
$$
Now $\{y \to \widetilde{\chi}(x -y) \psi_t(x-y),~ (t,x) \in K\}$ is an equicontinuous family of continuous functions supported inside a fixed compact set,  which implies that $(\widetilde{\chi}\psi_t) * \mu_m(x)$ converges to $(\widetilde{\chi} \psi_t) * \mu(x)$ uniformly on $K$.  It therefore follows that 
\begin{equation}\label{eqesp84}
\iint_K  |\psi_t *\mu(x)|d\mu(x)dt \leq
 \eta + \limsup_m \int_{1/m_0}^{m_0}
 \int_{x\in \bar B(0,m_0)} |({\widetilde{\chi}}\psi_t)*\mu_m(x)|d\mu_m(x)dt \leq 2\eta.
\end{equation}
Since $\eta$ is arbitrary the left side of \eqref{eqesp84} vanishes, and since  this holds for any $m_0\geq 1$, our claim \eqref{eqlimit}  proved.

Our next objective consists in showing that $\mu\in \wt\UU(\vphi,c_0)$. To this end, denote by $G$ the subset of those points $x\in\supp(\mu)$ such that
$$\int_{0}^\infty\!
 |\wt \Delta_{\mu,\vphi} (x,r)|\,dr=0.$$
It is clear now that $G$ has full $\mu$-measure. For $x\in G$, given $0<R_1< R_2$, we have
\begin{equation}\label{eqig2}
\begin{split}
|\vphi_{R_1} *\mu(x) - \vphi_{R_2}*\mu(x)| &= \left|\int_{R_1}^{R_2} r\,\partial_r \left(\vphi_r * \mu\right)(x)\,\frac{dr}r
\right|
\\
&
= \left|\int_{R_1}^{R_2} \wt \Delta_{\mu,\vphi} (x,r)\,\frac{dr}r\right|\\
&\leq 
\frac1{R_1} \int_{R_1}^{R_2} |\wt \Delta_{\mu,\vphi} (x,r)|\,dr=0.
\end{split}
\end{equation}
Therefore given $x,y\in G$ and $R_2 > R>0$ we have 
\begin{equation}\label{eqw2200}
|\vphi_R *\mu(x) - \vphi_{R}*\mu(y)|  =
|\vphi_{R_2} *\mu(x) - \vphi_{R_2}*\mu(y)| 
\leq \|\nabla(\vphi_{R_2} *\mu)\|_\infty \,|x-y|.
\end{equation}
Notice that
\begin{equation*}
\nabla(\vphi_{R_2}*\mu)(x) = \int \nabla\vphi_{R_2}(x-y)\,d\mu(y),
\end{equation*}
and by decomposing this integral into annuli centered at $x$, using the fast decay of $\nabla\vphi_{R_2}$ at $\infty$ and the fact that
$\mu(B(x,r))\leq c_0\,r^n$ for all $r>0$, we easily see  that
\begin{equation}\label{eqig1}
\|\nabla(\vphi_{R_2}*\mu)\|_\infty\leq \frac{c}{R_2},
\end{equation}
with $c$ depending on $c_0$.
Thus as $R_2\to\infty$ the right side of \rf{eqw2200} tends to $0$  and we conclude that
 $\vphi_R *\mu(x) =
 \vphi_{R}*\mu(y)$. 
 
 By continuity, since $G$ has full $\mu$ measure, it follows that
 $f(r,x) = \vphi_r*\mu(x)$ is constant on $(0,\infty)\times \supp(\mu)$. In other words, $\mu\in
 \wt\UU(\vphi,c_0)$.
 However, by condition \rf{equu12}, letting $m\to\infty$, we have 
$$\dist_{B(0,1)}(\mu,\wt\UU(\vphi,c_0)) \geq\ve,$$
because $\dist_{B(0,1)}(\cdot,\wt\UU(\vphi,c_0)) $ is continuous under the weak * topology, see \cite[Lemma 14.13]{Mattila}.
So $\mu\not\in\wt\UU(\vphi,c_0)$, which is a contradiction.
\end{proof}

By renormalizing the preceding lemma we get:
\begin{lemma}\label{lemcompactnorm1}
Let $\mu$ be an $n$-AD-regular measure such that $x_0\in\supp(\mu)$. For all
$\ve>0$ and $r>0$ there exists a constant $\delta>0$  such that
if 
$$\int_{\delta\,r}^{\delta^{-1}\,r}\!\!
 \int_{x\in \bar B(x_0,\delta^{-1}r)} |\wt \Delta_{\mu,\vphi} (x,t)|\,d\mu(x)\,dt \leq \delta\,r^{n+1},$$
then
$$\dist_{B(x_0,r)}(\mu,\wt\UU(\vphi,c_0)) <\ve\,r^{n+1}.$$
\end{lemma}

\begin{proof}
Let $T:\R^d\to\R^d$ be an affine map which maps $B(x_0,r)$ to $B(0,1)$. Consider the image measure
$\sigma=\frac1{r^n}\,T\#\mu$, where as usual $T\#\mu(E):= \mu (T^{-1}(E))$, and apply the preceding lemma to $\sigma$.
\end{proof}

\begin{definition}
\label{uni}
Given $n>0$, a Borel measure $\mu$ in $\R^d$ is called $n$-uniform if there exists a constant $c>0$
such that
$$\mu(B(x,r))=c\,r^n\quad\mbox{for all $x\in\supp(\mu)$ and $r>0$.}$$
\end{definition}

We will denote by $\UU(c_1)$ the collection of all $n$-uniform measures with constant $c_1$. 
By the following lemma, it turns out that $\wt\UU(\vphi,\cdot)$ and $\UU(\cdot)$ coincide.

\begin{lemma}\label{delellis}
Let $f:[0,\infty)\to[0,\infty)$ be defined either by $f(x) = e^{-x^{N}}$, for some $N\in\N$, or by $f(x) = (1 + x)^{-a}$, for some $a>n/2$. 
Let $\mu$ be a $n$-dimensional AD-regular  Borel measure  in $\R^{d}$.
 Then $\mu$ is $n$-uniform if and only if
there exists some constant $c>0$ such that
\begin{equation}\label{const_eq}
\int f\biggl({{|x-y|^2} \over {t^2}}\biggr) d\mu(y) =c\,t^n\quad\mbox{for all $x\in\supp(\mu)$ and $t>0$.}
\end{equation}
\end{lemma}

For $f(x) = e^{-x}$ this lemma is due to De Lellis (see pp. 60-61 of \cite{del}) and our proof closely follows his argument.

\begin{proof}  It is clear that \eqref{const_eq} holds if $\mu$ is $n$-uniform.  Now assume \eqref{const_eq}.
Write $Df(x) = x\,f'(x)$. Then
\begin{equation}\label{dense_eq} 
{\rm {span}}\bigl\{D^m f: m \geq 0\bigr\}\;
{\rm {~is ~dense ~in~}} L^1((0,\infty)),
\end{equation}
By the Weierstrass approximation theorem and our particular choice of $f$.

Let $\cB$ be the set of $g \in L^1((0,\infty))$ for which there is a constant $c_g$ such that 
$$
\int g\biggl({{|x-y|^2} \over {t^2}}\biggr) d\mu(y) = c_g t^n.
$$
Then $f \in \cB$, by the hypothesis \eqref{const_eq}.  Differentiating \eqref{const_eq} with respect to $t$ shows that $Df(x) = x\,f'(x) \in \cB$ with constant $-2cn$ independent of $x$. Then by induction and the assumption \eqref{dense_eq} $\cB$ contains a dense subset of $L^1((0,\infty))$.  Since $\cB$ is closed in $L^1((0,\infty))$, it follows that $\chi_{(0,1)} \in \cB$ and the lemma is proved.   
\end{proof}

\begin{lemma} \label{lemweak}
Let $\mu$ be an $n$-AD-regular measure in $\R^d$ such that $x_0\in\supp(\mu)$. For all
$\ve>0$, there exists a constant $\delta:=\delta(\ve)
>0$  such that
if, for some $r>0$,
$$\int_{\delta\,r}^{\delta^{-1}\,r}\!\!
 \int_{x\in \bar B(x_0,\delta^{-1}r)} |\wt \Delta_{\mu,\vphi} (x,t)|^2\,d\mu(x)\,\frac{dt}t \leq \delta^{n+4}\,r^n,$$
then there exists some constant $c_1>0$ such that
\begin{equation}
\label{wcdm}
|\mu (B(y,t))-c_1 t^n|<\ve r^n
\end{equation}
for all $y \in B(x_0,r) \cap \supp (\mu)$ and $0<t\leq r$.
\end{lemma}

\begin{proof} 
Let $\ve>0$. By Cauchy-Schwarz, we have
\begin{align*}
\int_{\delta\,r}^{\delta^{-1}\,r}\!\! &
 \int_{x\in \bar B(x_0,\delta^{-1}r)} |\wt \Delta_{\mu,\vphi} (x,t)|\,d\mu(x)\,dt\\
  & \leq \left[\int_{\delta\,r}^{\delta^{-1}\,r}\!\!
 \int_{x\in \bar B(x_0,\delta^{-1}r)} |\wt \Delta_{\mu,\vphi} (x,t)|^2\,d\mu(x)\,\frac{dt}t\right]^{1/2} 
 \left[\int_{\delta\,r}^{\delta^{-1}\,r}\!\!
 \int_{x\in \bar B(x_0,\delta^{-1}r)} t\,d\mu(x)\,dt\right]^{1/2}\\
&\\
&\leq c \bigl[\delta^{n+4}\,r^n\bigr]^{1/2}\,\bigl[\delta^{-2}\,r^2\,\mu(B(x_0,\delta^{-1}r))\bigr]^{1/2}\\
 &\\
& \leq c\, \bigl[\delta^{(n+4)/2}\,r^{n/2}\bigr]\,\bigl[\delta^{-(n+2)/{2}}\,r^{(n+2)/2}\bigr] = c\, \delta\,r^{n+1}.
\end{align*}

Hence for any $\ve_1 > 0$  we see that if $\delta$ is small enough  then by Lemma 3.2,  
$$\dist_{B(x_0,3r)}(\mu,\wt\UU(\vphi,c_0)) <\ve_1\,r^{n+1}$$ 
and  there exists $\sigma \in \UU(c_1)$ such that $\dist_{B(x_0,3r)}(\mu,\sigma) <\ve_1\,r^{n+1}$ for a suitable constant $c_1$.

Let $y \in B(x_0,r)$ and for $0<s \leq r$ consider a smooth bump function $\wt\chi_{y,s}$ such that $ \chi_{B(y,s)}  \leq \wt \chi_{y,s} \leq \chi_{B(y,s(1+\eta))}$ and $\|\nabla \wt\chi_{y,s}\|_\infty \leq \frac c{s\eta}$, where $\eta$ is some small constant to be determined later.
For $y \in B(x_0,r)$ and for $0<s \leq r$, we have
\begin{equation}
\label{intsm}
\begin{split}
&\left| \int \wt\chi_{y,s} (x)d\mu(x)-\int \cys (x) d \sigma (x)\right|\\
&\leq \|\nabla \wt\chi_{y,s} \|_\infty \,\dist_{B(x_0,3r)}(\mu,\sigma) \leq c\frac{\ve_1\,r^{n+1}}{\eta \, s}.
\end{split}
\end{equation}
Therefore by \eqref{intsm} and Lemma \ref{delellis}, for $0<t\leq r$,
\begin{equation}
\begin{split}
\mu(B(y,t)) &\leq \int \wt\chi_{y,t}(x)\, d \mu (x) \leq  \int \wt\chi_{y,t}(x)\, 
d \sigma (x) + c\frac{\ve_1\,r^{n+1}}{\eta \, t} \\
&\leq c_1 t^{n}(1+\eta)^n+c\frac{\ve_1 \,r^{n+1}}{\eta \, t}, 
\end{split}
\end{equation}
and
\begin{equation}
\begin{split}
\mu(B(y,t)) &\geq \int \wt\chi_{y,\frac t{1+\eta}}(x)\, d \mu (x) \geq  \int \wt\chi_{y,\frac t{1+\eta}}(x)\, d \sigma (x) - c\frac{\ve_1\,r^{n+1}}{\eta\, t} \\
&\geq c_1 \frac{t^{n}}{(1+\eta)^n}-c\frac{\ve_1\,r^{n+1}}{\eta\, t}. 
\end{split}
\end{equation}
Choosing $\eta$ and $\ve_1$ appropriately,  we get that for some small $\ve_2:=\ve_2(\ve_1, \eta)$,
\begin{equation}
|\mu (B(y,t))-c_1 t^n|\leq \ve_2 \left( \frac{r^{n+1}}{t}+ t^n \right).
\end{equation}
Hence if $t> {\ve_2}^{1/2}r$, then because  $t^n \leq r^{n+1}
/t$,
$$|\mu (B(y,t))-c_1 t^n| \leq c\,{\ve_2} \,\frac{r^{n+1}}{{\ve_2}^{1/2}r}\leq c\, \ve_2^{1/2}r^n.$$
On the other hand, if $t \leq {\ve_2}^{1/2}r$, then by the AD-regularity of $\mu$,
$$|\mu (B(y,t))-c_1 t^n| \leq \mu(B(y,t))+c_1 t^n \leq c ({\ve_2}^{1/2})^n\,r^n.$$
Therefore, since $\lim_{\ve_1 \to 0, \eta \to 0} \ve_2 =0$,  \eqref{wcdm} holds if $\ve_1$ and $\eta$ are
sufficiently small. 
\end{proof}

\begin{lemma}\label{lemcarleson}
Let $\mu$ be an $n$-AD-regular measure. Assume that $|\wt \Delta_{\mu,\vphi} (x,r)|^2\,d\mu(x)\,\frac{dr}r$ is a Carleson measure on $\supp(\mu)\times(0,\infty)$. Then the weak constant density condition holds for $\mu$.
\end{lemma}

\begin{proof}
Let $\ve>0$ and let $A:=A_\ve\subset \R^d\times \R$ consist of those pairs $(x,r)$ such that \eqref{wcdm} does not hold.
We have to show that
$$\int_0^R \int_{x\in B(z,R)} \chi_A(x,r)\,d\mu(x)\,\frac{dr}r \leq c(\ve)\,R^n
\quad \mbox{ for all $z\in\supp(\mu)$, $r>0$.}$$
To this end, notice that if $(x,r)\in A$, then
$$\int_{\delta\,r}^{\delta^{-1}\,r}\!\!
 \int_{y\in \bar B(x,\delta^{-1}r)} |\wt \Delta_{\mu,\vphi} (y,t)|^2\,d\mu(y)\,\frac{dt}t \geq \delta^{n+4}\,r^n,$$
where $\delta=\delta(\ve)$ is as in Lemma \ref{lemweak}.
Then by Chebychev's inequality,
\begin{align*}
\int_0^R \int_{x\in B(z,R)} &\chi_A(x,r)\,d\mu(x)\,\frac{dr}r \\
& \leq
\int_0^R \int_{x\in B(z,R)} \frac{1}{\delta^{n+4}\,r^n}\left(
\int_{\delta\,r}^{\delta^{-1}\,r}\!\!
 \int_{y\in \bar B(x,\delta^{-1}r)} |\wt \Delta_{\mu,\vphi} (y,t)|^2\,d\mu(y)\,\frac{dt}t\right)
\,d\mu(x)\,\frac{dr}r\\
& \leq \int_0^{\delta^{-1}R} \int_{|y-z|\leq (1+\delta^{-1})R} |\wt \Delta_{\mu,\vphi} (y,t)|^2\,
\int_{\delta \,t}^{\delta^{-1}t}
\frac{\mu(B(y,\delta^{-1}r))}{\delta^{n+4}\,r^{n+1}}  
 dr\,
\,d\mu(y)\,\frac{dt}t.
\end{align*}
But since
$$\int_{\delta \,t}^{\delta^{-1}t}
\frac{\mu(B(y,\delta^{-1}r))}{\delta^{n+4}\,r^{n+1}}  
 dr\leq c_0\,\delta^{-2(n+2)}
\int_{\delta \,t}^{\delta^{-1}t}
\frac{dr}{r}  \leq c_0\,\delta^{-2(n+3)},
$$
we then get
\begin{equation*}
\begin{split}
\int_0^R\! \int_{x\in B(z,R)}\! & \chi_A(x,r)\,d\mu(x)\,\frac{dr}r \\
&\leq c_0\,\delta^{-2(n+3)}\!\!
\int_0^{\delta^{-1}R}\! \!\int_{|y-z|\leq (1+\delta^{-1})R} |\wt \Delta_{\mu,\vphi} (y,t)|^2
\,d\mu(y)\,\frac{dt}t\leq c\,\delta^{-2n-7}R^n,
\end{split}
\end{equation*}
which is what we needed to show. 
\end{proof}

As an immediate corollary of Theorem \ref{wcdur} and Lemma \ref{lemcarleson} we obtain the following.

\begin{theorem}
\label{main2od2}
If $\mu$ is an $n$-AD-regular measure in $\R^d$ and if $c$ is a constant such that for any ball $B(x_0,R)$ 
with center $x_0 \in \supp(\mu)$,
$$\int_0^R \int_{x\in B(x_0,R)} |\wt \Delta_{\mu,\vphi}(x,r)|^2\,d\mu(x)\,\frac{dr}r \leq c\, R^n,$$
then $\mu$ is uniformly $n$-rectifiable.
\end{theorem}




We denote by $\UU(\vphi,c_0)$ the family of $n$-AD-regular measures with constant $c_0$ in $\R^d$ such that 
$$\Delta_{\mu,\vphi} (x,r) = 0 \quad\mbox{for all $r>0$ and all $x\in\supp(\mu)$}.$$
By an argument similar to  the proof of Lemma \ref{lemcompact}
we obtain the following.

\begin{lemma}\label{lemcompactnorm}
Let $\mu$ be an $n$-AD-regular measure with constant $c_0$ in $\R^d$ such that $x_0\in\supp(\mu)$. For all
$\ve>0$ and $r>0$ there exists a constant $\delta>0$ such that
if 
$$\int_{\delta\,r}^{\delta^{-1}\,r}\!\!
 \int_{x\in \bar B(x_0,\delta^{-1}r)} |\Delta_{\mu,\vphi} (x,t)|\,d\mu(x)\,dt \leq \delta\,r^{n+1},$$
then
$$\dist_{B(x_0,r)}(\mu,\UU(\vphi,c_0)) <\ve\,r^{n+1}.$$
\end{lemma}

The details of the proof are left for the reader.

\begin{lemma}\label{lemmaunif}
If $\mu\in \UU(\vphi,c_0)$ then $\mu$ is supported on an $n$-rectifiable set.
\end{lemma}

\begin{proof}

Since $\mu\in \UU(\vphi,c_0)$ we have
\begin{equation}\label{eqkw433}
\vphi_{2^{-k}} *\mu(x) - \vphi_{2^k}*\mu(x) =0\qquad \mbox{ for all $k>0$ and all $x\in\supp(\mu)$.}
\end{equation}
Now consider the function $F:\R^d\to \R$ defined by
$$F(x) = \sum_{k>0}2^{-k}\Bigl(\vphi_{2^{-k}} *\mu(x) - \vphi_{2^k}*\mu(x)\Bigr)^2.$$
Taking into account that $|\vphi_{2^{-k}} *\mu(x) - \vphi_{2^k}*\mu(x)|\leq c$ for all $x\in\R^d$ and $k\in\N$, it is clear 
that $F(x)<\infty$ for all $x\in\R^d$, and so $F$ is well defined. 
Moreover, by \rf{eqkw433} we have $F =0$ on  $\supp(\mu)$. 

Now we claim that $F(x)>0$ for all $x\in\R^d\setminus \supp(\mu)$.
Indeed, it follows easily that
$$\lim_{k\to\infty}\vphi_{2^{-k}}*\mu(x)= 0\qquad \mbox{for all $x\in\R^d\setminus \supp(\mu)$,}$$
while, by the $n$-AD-regularity of $\mu$,
$$\liminf_{k\to\infty}\vphi_{2^{k}}*\mu(x)\geq c\, c_0^{-1}\qquad \mbox{for all $x\in\R^d$.}$$
Thus if $x\in\R^d\setminus \supp(\mu)$ we have $\vphi_{2^{-k}} *\mu(x) - \vphi_{2^k}*\mu(x)\neq0$ for all large $k>0$,
 which implies that $F(x)>0$ and proves our claim.

We have shown that for $\mu\in \UU(\vphi,c_0)$,  $\supp(\mu) = F^{-1}(0)$. Next we will show $F^{-1}(0)$  is a real analytic variety. Notice that the lemma will follow
from this assertion because   $\supp(\mu)$ has locally finite $\HH^n$ measure,
so that the analytic variety $F^{-1}(0)$ is $n$-dimensional and any $n$-dimensional real analytic variety is $n$-rectifiable.

To prove that the zero set of $F$ is a real analytic variety, it is enough to check that
$\vphi_{2^{-k}} *\mu - \vphi_{2^k}*\mu$ is a real analytic function for each $k>0$, because the
zero set of a real analytic function is a real analytic variety and the intersection
of any family  of real analytic varieties is again a real analytic variety; see \cite{nar}.   So it is
enough to show that $\vphi_r *\mu$ is a real analytic function for every $r>0$. 

In the case $\vphi(x)=e^{-|x|^{2N}}$, consider the function $f:\C^d\to\C$ defined by
$$f(z_1,\ldots,z_d)= \frac1{r^n}\int \exp\Biggl(-r^{-2N}\biggl(\sum_{i=1}^d (y_i-z_i)^2\biggr)^{N}\Biggr)\,d\mu(y).$$
It is easy to check that $f$ is well defined and holomorphic in the whole $\C^d$, and thus
$\vphi_r *\mu = f|_{\R^d}$ is real analytic.

In the case $\vphi(x)= (1+|x|^2)^{-a}$, $a>n/2$, for $(z_1,\ldots,z_d)\in\C^d$
we take 
$$
f(z_1,\ldots,z_d)= \frac1{r^n}\int \biggl(1 +r^{-2}\sum_{i=1}^d (y_i-z_i)^{2}\biggr)^{-a}\,d\mu(y).
$$
This is a holomorphic function in the open set
$$V=\Bigl\{z\in\C^d:|{\rm Im}\,z_i|<\frac r{2d^{1/2}}\;\text{for $1\leq i \leq d$}\Bigr\}.$$
Indeed, for $z\in V$, we have
$${\rm Re}\,\biggl(1 +r^{-2}\sum_{i=1}^d (y_i-z_i)^{2}\biggr)
= 1 + r^{-2}\sum_{i=1}^d \bigl((y_i-{\rm Re}\,z_i)^{2} - ({\rm Im}\,z_i)^{2} \bigr)
\geq 1 - r^{-2}\sum_{i=1}^d ({\rm Im}\,z_i)^{2} > \frac34.$$
Thus $f$ is well defined and holomorphic in $V$, and so $\vphi_r *\mu = f|_{\R^d}$ is real analytic.
\end{proof}
\vvv

\begin{theorem} 
\label{teounif}
If $\mu \in \UU(\vphi,c_0)$ then $\mu$ is $n$-uniform.
\end{theorem}
\begin{proof}
If $\mu \in \UU(\vphi, c_0)$, then 
$\vphi_r*\mu = \vphi_{2r}*\mu(x)$ for all $x\in\supp(\mu)$ and all $r>0$, and consequently  
\begin{equation}\label{eqda211}
\vphi_{2^k r} *\mu(x) = \vphi_r*\mu(x)\quad\mbox{\; for all
$1\leq r<2$, all $k\in\Z$, and all $x\in\supp(\mu)$}.
\end{equation}
By the preceding lemma $\mu$ is of the form
$$\mu = \rho\,\HH^n\rest E,$$
where $\rho$ is some positive function on $E$ bounded from above and below and $E\subset\R^d$ 
is an $n$-rectifiable set. This implies that the density
$$\Theta^n(x,\mu)=\lim_{\ve\to0}\frac{\mu(B(x,\ve))}{(2\ve)^n}$$
exists at $\mu$-a.e. $x\in\R^d$; see \cite[Theorem 16.2]{Mattila}.
It then  follows easily that
$$\lim_{\ve\to 0}\vphi_\ve*\mu(x) \quad\mbox{exists at $\mu$-a.e.\ $x\in\R^d$}$$
and with \rf{eqda211} this implies that 
$$\vphi_{R_1}*\mu(x)= \vphi_{R_2}*\mu(x)\quad\mbox{\; for all $R_1,R_2>0$ and $\mu$-a.e.\ $x\in\R^d$.}$$
Using an argument analogous to the proof of  Lemma \ref{lemcompact} we then conclude that 
$$\vphi_{R_1}*\mu(x)= \vphi_{R_2}*\mu(y)\quad\mbox{\; for all $R_1,R_2>0$ and all
$x,y\in\supp(\mu)$}.$$
Therefore, by Lemma \ref{delellis}, $\mu$ is $n$-uniform.
\end{proof}

Using Lemma \ref{lemcompactnorm} and Theorem \ref{teounif}, we can, with minor changes in
their proofs, obtain analogues of Lemmas \ref{lemweak} and \ref{lemcarleson} with $\wt\Delta_{\mu,\vphi}$ replaced by $\Delta_{\mu,\vphi}$. Hence we concluded the following.

\begin{theorem}
\label{main2od1}
If $\mu$ is an $n$-AD-regular measure in $\R^d$ and there exists a constant $c$ such that for any ball $B(x_0,R)$ centered at $\supp(\mu)$
$$\int_0^R \int_{x\in B(x_0,R)} |\Delta_{\mu,\vphi}(x,r)|^2\,d\mu(x)\,\frac{dr}r \leq c\, R^n,$$
then $\mu$ is uniformly $n$-rectifiable.
\end{theorem}


\begin{coro}\label{coro***}
Suppose that for any ball $B(x_0,R)$ centered at $\supp(\mu)$
\begin{equation*}
\int_0^R \int_{x\in B(x_0,R)} |\Delta_\mu(x,r)|^2\,d\mu(x)\,\frac{dr}r \leq c\, R^n.
\end{equation*}
Then $\mu$ is uniformly $n$-rectifiable.
\end{coro}

\begin{proof}
We will show that \eqref{eqsq2} implies \rf{eqsqdif}, by taking a suitable convex combination,
and then apply Theorem \ref{main2od1}. 

For $R>0$ we seek a function $\wt \vphi_R:(0,\infty) \ra (0,\infty)$ such that
\begin{equation}
\label{convfi}
\frac{1}{R^n} e^{\frac{-s^2}{R^2}}=\int_0^\infty \frac1{r^n}\chi_{[0,r]}(s) \wt \vphi_R(r) \,dr=\int_s^\infty \frac{\wfi(r)}{r^n}\,dr,\quad \mbox{ for } s>0.
\end{equation}
Differentiating with respect to $s$ we get
$$-\frac{2 s}{R^{n+2}}\,e^{\frac{-s^2}{R^2}}=- \frac{\wfi(s)}{s^n}.$$
Hence \rf{convfi} is solved for $R>0$ and  $s>0$ by 
$$\wfi(s)=\frac{2 s^{n+1}}{R^{n+2}}e^{\frac{- s^2}{R^2}}.$$

Using \rf{convfi} we can now write, for $x \in \supp (\mu)$, and any $R_1>0$,
\begin{equation*}
\begin{split}
&\int_0^{\infty} |\Delta_{\mu,\vphi}(x,R)|^2 \, \frac{dR}{R}=\int_0^{\infty} |(\vphi_R -\vphi_{2R})\ast \mu(x)|^2 \,\frac{dR}{R}\\
&\;=\int_0^{\infty} \Big| \left( \int_0^\infty \!\frac 1{r^n} \chi_{[0,r]}(|\cdot |) \wfi (r)\,dr)\right) \ast \mu (x) -\left( \int_0^\infty \frac 1{r^n} \chi_{[0,r]}(|\cdot |) \wf_{2R} (r)\,dr\right)\ast \mu (x)\Big|^2 \, \frac{dR}R.
\end{split}
\end{equation*}
By a change of  variables we get
$$\int_0^\infty \frac 1{r^n} \chi_{[0,r]}(|y-x|) \wf_{2R} (r)\,dr= \int_0^\infty \!\frac 1{(2r)^n} \chi_{[0,2r]}(|y-x|) \wfi(r)\, dr.$$
Therefore, using Cauchy-Schwarz and the fact that $\int_0^\infty\wt\vphi_R(r)\,dr\lesssim1$,
we obtain
\begin{equation*}
\begin{split}
\int_0^{\infty} |\Delta_{\mu,\vphi}(x,R)|^2 \, \frac{dR}{R}&=\int_0^{\infty} \left| \int_0^{\infty} \left( \frac1{r^n} \chi_{B(0,r)}(\cdot)-\frac1{(2r)^n} \chi_{B(0,2r)}(\cdot)\right) \ast \mu (x) \, \wfi(r)\, dr \right|^2 \, \frac{dR}{R}\\
&\lesssim \int_0^{\infty}  \int_0^\infty |\Delta_\mu (x,r)|^2 \wfi(r)\, dr\, \frac{dR}{R}\\
&\lesssim \int_0^{\infty} \left( \int_0^\infty \wfi (r) \, \frac{dR}{R} \right)|\Delta_\mu (x,r)|^2 \, dr.
\end{split}
\end{equation*}
Moreover,
\begin{equation*}
\int_0^\infty \wfi (r) \, \frac{dR}{R}=2 \int_0^\infty \left(\frac{r}{R}\right)^{n+1} e^{\frac{-  r^2}{R^2}}\, \frac{dR}{R^2}=\frac2r  \int_0^\infty  t^{n+1}e^{- t^2}\,dt \lesssim \frac1r.
\end{equation*}
Hence we infer that
$$\int_0^{\infty} |\Delta_{\mu,\vphi}(x,r)|^2 \, \frac{dr}{r} \lesssim \int_0^{\infty} |\Delta_{\mu}(x,r)|^2 \, \frac{dr}{r},$$
which shows that \eqref{eqsq2} implies \rf{eqsqdif}.
\end{proof}




\section{Uniform rectifiabilty implies boundedness of smooth square functions}\label{sec:urbddsm}

Let $h:\R^d \ra \R$ be a smooth function for which there exist positive constants $c$ and $\ve$ such that
\begin{equation}
\label{hprop}
|h(x)| \leq \frac{c}{(1+|x|)^{n+\ve}}\quad\text{ and }\quad|\nabla h (x)| \leq \frac{c}{(1+|x|)^{n+1+\ve}},
\end{equation}
for all $x \in \Rd$. Furthermore assume that $$\int h(y-x) d \ha^n_{|L}(y)=0$$
for every $n$-plane $L$ and every $x \in L$. For $r>0$, denote 
$$h_r(x) = \frac1{r^n}\,h\left(\frac xr\right).$$

\begin{theorem}\label{teoq1}
Let $\mu$ be an $n$-AD-regular measure in $\R^d$.
If $\mu$ is uniformly $n$-rectifiable, then there exists a constant $c$ such that
\begin{equation}\label{eqsq11}
\int_0^R \int_{x\in B(x_0,R)}|h_r *\mu|^2\,d\mu(x)\,\frac{dr}r \leq c\, R^n,
\end{equation}
for all $x_0\in\supp(\mu)$, $R>0$.
\end{theorem}

\vvv

\begin{proof}
It is immediate to check that the estimate \rf{eqsq11} holds if and only if for all $R_0\in\DD$
\begin{equation}\label{eqsq14}
\sum_{Q\in\DD:Q\subset R_0} \int_Q \int_{\ell(Q)}^{2\ell(Q)} |h_r*\mu(x)|^2\,\frac{dr}{\ell(Q)}\,d\mu(x)
\leq c\,\mu(R_0).
\end{equation}

Let $x\in \frac12 B_Q$ and $\ell(Q)\leq r\leq 2\ell(Q)$. If $x\in \frac12 B_Q \cap L_Q$ (recall that $L_Q$ is the $n$-plane 
minimizing $\alpha(Q)$), we have
$$\int h_r(y-x)\,d\HH^n_{|L_Q}(y)=0.$$ 
Hence
\begin{align*}
\left|\int h_r(y-x)\,d\mu(y)\right| & = \left|\int h_r(y-x)\,d (\mu - c_Q\HH^n_{|L_Q})(y)\right| \\
&= \left|\int \sum_{k \geq 0} \wt\chi_k (y)\, h_r(y-x)\,d (\mu - c_Q\HH^n_{|L_Q})(y)\right|\\
&\leq \sum_{k \geq 0} \left| \int \wt\chi_k(y)\,h_r(y-x) \, d (\mu - c_Q\HH^n_{|L_Q})(y)\right|,
\end{align*}
where $\wt\chi_k$, $k\geq0$, are bump smooth functions such that 
\begin{itemize}
\item $\sum_{k\geq 0} \wt\chi_k=1$
\medskip
\item $\|\nabla \wt\chi_k\|_\infty \leq \ell(Q^k)^{-1}$, 
\medskip
\item $\chi_{A(x,2^{k}\,r,2^{k+1}\,r)} \leq \wt \chi_k \leq  \chi_{A(x,2^{k-1}\,r,2^{k+2}\,r)}$ for $k \geq 1$, and
\medskip
\item $\chi_{B(x,r)}\leq\wt\chi_0 \leq \chi_{B(x,2r)}.$
\end{itemize}
As usual $A(x,r_1,r_2)=\{y:r_1\leq|y-x|<r_2\}$. Moreover for $m \in \N$, $Q^{m}$ denotes the ancestor of $Q$ such that $\ell(Q^{m})=2^{m} \ell(Q)$.

Set $F_k(y)= h_r(x-y) \wt \chi_k(y)$, and notice that $\supp F_k \subset B_{Q^{k+2}}$. Then 
\begin{equation}
\label{hrest}
\begin{split}
\left|\int h_r(y-x)\,d\mu(y)\right| &\leq \sum_{k\geq 0} \left| \int F_k(y) d (\mu - c_{Q^{k+2}}\HH^n_{|L_{Q^{k+2}}})(y)\right| \\
&\quad\quad \quad +\sum_{k\geq 0} \left| \int F_k(y) d ( c_Q\HH^n_{|L_Q}-c_{Q^{k+2}}\HH^n_{|L_{Q^{k+2}}})(y)\right|\\
&\leq  \sum_{k \geq 0} \| \nabla F_k\|_\infty \,\a(Q^{k+2})\, \ell(Q^{k+2})^{n+1}\\
&\quad\quad \quad+ \sum_{k \geq 0} \| \nabla F_k\|_\infty\,\dist_{B_{Q^{k+2}}}(c_Q\HH^n_{|L_Q},c_{Q^{k+2}}\HH^n_{|L_{Q^{k+2}}})\\
&:=I_1+I_2
\end{split}
\end{equation}

For $y \in \supp F_k$ using \eqref{hprop} it follows easily that
$$|h_r(y-x)|\lesssim \frac1{\ell(Q)^n}\left(\frac{\ell(Q)}{\ell(Q^k)} \right)^{n+\ve} \;\text{ and }\quad| \nabla h_r(y-x)| \lesssim  \frac1{\ell(Q)^{n+1}}\left(\frac{\ell(Q)}{\ell(Q^k)} \right)^{n+1+\ve}.$$
Hence
\begin{equation}
\begin{split}
\label{fkbound}
\| \nabla F_k\|_\infty &\lesssim \frac1{\ell(Q^k)}\frac1{\ell(Q)^n}\left(\frac{\ell(Q)}{\ell(Q^k)} \right)^{n+\ve}+\frac1{\ell(Q)^{n+1}}\left(\frac{\ell(Q)}{\ell(Q^k)} \right)^{n+1+\ve}\lesssim \frac{\ell(Q)^\ve}{\ell(Q^k)^{n+1+\ve}}.
\end{split}
\end{equation}

We can now estimate $I_1$:
\begin{equation}
\label{i1sm}
\begin{split}
I_1 &\lesssim \sum_{k \geq 0} \a (Q^{k+2}) \ell(Q^k)^{n+1} \frac{\ell(Q)^\ve}{\ell(Q^k)^{n+1+\ve}}=\sum_{k \geq 0} \a (Q^{k+2}) \left( \frac{\ell(Q)}{\ell(Q^k)} \right)^\ve\\
&\lesssim \sum_{P \in \DD: R \supset Q } \a(P)\left( \frac{\ell(Q)}{\ell(P)} \right)^\ve.
\end{split}
\end{equation}

For $I_2$, using also \cite[Lemma 3.4]{Tolsa-lms}, we get
\begin{equation}
\label{i2sm}
\begin{split}
I_2 &\lesssim  \sum_{k \geq 0}  \frac{\ell(Q)^\ve}{\ell(Q^k)^{n+1+\ve}} \Biggl( \sum_{0 \leq j \leq k+2} \a (Q^j) \Biggr) \ell (Q^{k+2})^{n+1} \\
&\lesssim \sum_{k \geq 0} \Biggl( \frac{\ell(Q)}{\ell(Q^k)} \Biggr)^\ve \Biggl( \sum_{0 \leq j \leq k+2} \a (Q^j) \Biggr) \\
&\lesssim \sum_{R \in \DD: R \supset Q} \, \sum_{P \in \DD: Q \subset P \subset R} \a(P) \left( \frac{\ell(Q)}{\ell(R)} \right)^\ve \\
&= \sum_{P \in \DD : P \supset Q} \a (P) \sum_{R \in \DD: R \supset P} \left( \frac{\ell(Q)}{\ell(R)} \right)^\ve \\
&\approx \sum_{P \in \DD : P \supset Q} \a (P) \left( \frac{\ell(Q)}{\ell(P)} \right)^\ve
\end{split}
\end{equation}

Therefore by \eqref{hrest}, \eqref{i1sm} and \eqref{i2sm}, for $x\in \frac12 B_Q \cap L_Q$ and $\ell(Q) \leq r \leq 2 \ell(Q)$,
\begin{equation}
\label{hrestl}
\left|\int h_r(y-x)\,d\mu(y)\right| \lesssim \sum_{P \in \DD : P \supset Q} \a (P) \left( \frac{\ell(Q)}{\ell(P)} \right)^\ve.
\end{equation}

On the other hand, given an arbitrary $x\in Q$, let $x'$ be its orthogonal projection on $L_Q$ (notice 
that $x'\in\frac12 B_Q$). We have
\begin{equation}
\label{hrpw}
\begin{split}
\left|\int h_r(y-x)\,d\mu(y)\right| &\leq \left|\int h_r(y-x')\,d\mu(y)\right|+\int_{B_Q} |h_r(y-x)-h_r(y-x')|\,d\mu(y) \\
&\quad \quad \quad +\int_{\R^d \stm B_Q} |h_r(y-x)-h_r(y-x')|\,d\mu(y)\\
&:=I_3+I_4+I_5.
\end{split}
\end{equation}
For $\ell(Q) \leq r \leq 2 \ell(Q)$, by \eqref{hrestl},
\begin{equation}
\label{i3sm}
I_3 \lesssim \sum_{P \in \DD : P \supset Q} \a (P) \left( \frac{\ell(Q)}{\ell(P)} \right)^\ve.
\end{equation}

We can now estimate $I_4$ and $I_5$ using \eqref{hprop}. First
\begin{equation}
\label{i4sm}
I_4 \lesssim \int_{B_Q} \frac{|x-x'|}{ \ell(Q)^{n+1}}  d \mu (y) \lesssim \frac{\dist(x, L_Q)}{ \ell(Q)^{n+1}} \ell(Q)^n= \frac{\dist(x,L_Q)}{ \ell(Q)}.
\end{equation}
Moreover, noticing that if $y \notin B_Q$ and $\xi \in [y-x,y-x']$ we have that $|y-x| \approx |\xi|$,
\begin{equation}
\label{i5sm}
\begin{split}
I_5 &\lesssim \int \frac{|x-x'|}{ \ell(Q)^{n+1}} \sup_{ \xi \in [y-x,y-x']} |\nabla (h_r) ( {\xi})|d \mu (y) \\
&\lesssim \frac{|x-x'|}{ \ell(Q)^{n+1}} \int \frac{\ell(Q)^{n+1+\ve}}{(\ell(Q)+|y-x|)^{n+1+\ve}} d \mu (y) \\
&\lesssim \dist(x,L_Q) \ell(Q)^\ve \ell (Q)^{-1-\ve} =\frac{\dist(x,L_Q)}{\ell(Q)}.
\end{split}
\end{equation}
Hence by \eqref{hrpw}, \eqref{i3sm}, \eqref{i4sm} and \eqref{i5sm}, we get the following pointwise estimate for $x \in Q$ and $\ell(Q) \leq r \leq2\ell(Q)$:
\begin{equation}
|h_r*\mu(x)| \lesssim \frac{\dist(x,L_Q)}{\ell(Q)}+\sum_{P \in \DD : P \supset Q} \a (P) \left( \frac{\ell(Q)}{\ell(P)} \right)^\ve.
\end{equation}
Therefore,
\begin{equation*}
\begin{split}
&\sum_{Q \in \DD:Q \subset R_0} \int_Q \int_{\ell(Q)}^{2\ell(Q)} |h_r*\mu(x)|^2\,\frac{dr}{\ell(Q)} d \mu (x)\\
&\quad\lesssim \sum_{Q \in \DD:Q \subset R_0} \int_Q \int_{\ell(Q)}^{2\ell(Q)} \left(\frac{\dist(x,L_Q)}{\ell(Q)}\right)^2 \frac{dr}{\ell(Q)} d \mu (x)\\
&\quad\quad\quad+\sum_{Q \in \DD:Q \subset R_0} \int_Q \int_{\ell(Q)}^{2\ell(Q)} \left( \sum_{P \in \DD : P \supset Q} \a (P) \left( \frac{\ell(Q)}{\ell(P)} \right)^\ve \right)^2 \frac{dr}{\ell(Q)}d\mu(x)\\
&\quad \lesssim \sum_{Q \in \DD:Q \subset R_0} \int \left(\frac{\dist(x,L_Q)}{\ell(Q)}\right)^2 d \mu(x)+ \\
&\quad\quad\quad+ \sum_{Q \in \DD:Q \subset R_0} \left( \sum_{P \in \DD: P \supset Q} \a(P)^2 \left( \frac{\ell(Q)}{\ell(P)}\right)^{\ve}\right) \left( \sum_{P \in \DD: P \supset Q} \left( \frac{\ell(Q)}{\ell(P)}\right)^{\ve}\right) \mu (Q),
\end{split}
\end{equation*}
where we used Cauchy-Schwarz for the last inequality. By \cite[Lemmas 5.2 and 5.4]{Tolsa-lms},
\begin{equation*}
\sum_{Q\in\DD:Q\subset R_0} \int_Q \frac{\dist(x,L_Q)^2}{\ell(Q)^2}\,d\mu(x) \lesssim\mu(R_0).
\end{equation*}
Finally,
\begin{equation*}
\begin{split}
\sum_{Q \in \DD:Q \subset R_0}& \left( \sum_{P \in \DD: P \supset Q} \a(P)^2 \left( \frac{\ell(Q)}{\ell(P)}\right)^{\ve}\right) \left( \sum_{P \in \DD: P \supset Q} \left( \frac{\ell(Q)}{\ell(P)}\right)^{\ve}\right) \mu (Q)\\
&\lesssim \sum_{Q\in \DD: Q \subset R_0}\,\sum_{P \in \DD: Q \subset P \subset R_0} \a(P)^2 \left( \frac{\ell(Q)}{\ell(P)}\right)^{\ve} \mu(Q) \\
&\quad\quad+\sum_{Q\in \DD: Q \subset R_0}\,\sum_{P \in \DD: P \supset R_0} \a(P)^2 \left( \frac{\ell(Q)}{\ell(P)}\right)^{\ve} \mu(Q)\\
&\lesssim \sum_{P\in \DD: P \subset R_0} \a(P)^2 \sum_{Q \in \DD: Q \subset P}  \left( \frac{\ell(Q)}{\ell(P)}\right)^{\ve} \mu(Q) +\sum_{Q\in \DD: Q \subset R_0} \left( \frac{\ell(Q)}{\ell(R_0)}\right)^{\ve} \mu(Q) \\
&\lesssim \sum_{P\in \DD: P \subset R_0} \a(P)^2 \mu (P)+ \mu (R_0)\lesssim \mu(R_0),
\end{split}
\end{equation*}
where the last inequality follows from Theorem \ref{teotol}.
\end{proof}

The proof of Theorem \ref{main2} follows from Theorems \ref{main2od1}, \ref{main2od2} and Theorem \ref{teoq1}.


\section{Uniform rectifiabilty implies boundedness of square functions: the non-smooth case}\label{sec:urbdd}

By Corollary \ref{coro***} we already know that condition \rf{eqsq2} implies the uniform
$n$-rectifiability of $\mu$, assuming $\mu$ to be $n$-AD-regular. So to complete the proof of Theorem
\ref{main1} it remains to show that \rf{eqsq2} holds for any ball $B(x_0,R)$ centered at $\supp(\mu)$ if $\mu$ is uniformly $n$-rectifiable.
To this end, we would like to argue as in the preceding section, setting
$$\phi_r = \frac1{r^n}\,\chi_{B(0,r)}(x),\qquad x\in\R^d,$$
and
$$h_r = \phi_r - \phi_{2r}.$$
The main obstacle is the lack of smoothness of $h_r$.
To solve this problem we will decompose $h_r$ using wavelets as follows.

Consider a family of $C^1$ compactly supported orthonormal wavelets in $\R^n$. Tensor products of Daubechies compactly supported wavelets with $3$ vanishing moments will suffice for our purposes, see e.g. \cite[Section 7.2.3]{mal}. We denote this family of functions by
$\{\psi_I^\epsilon\}_{I\in\DD(\RR^n),1\leq\epsilon\leq 2^n-1}$, where $\DD(\R^n)$ is the standard
grid of dyadic cubes in  $\R^n$. Each $\psi_I^\epsilon$ is a $C^1$ function supported on $5I$, which satisfies
$\|\psi_I^\epsilon\|_2=1$, and moreover
$$\|\psi_I^\epsilon\|_\infty\lesssim \frac1{\ell(I)^{n/2}},\qquad \|\nabla \psi_I^\epsilon\|_\infty\lesssim \frac1{\ell(I)^{1+n/2}}
\qquad\text{for all $I\in\DD(\R^n)$ and $1\leq\epsilon\leq 2^n-1$,}$$
where $\ell(I)$ is the sidelength of the cube $I$. Recall that any function $f\in L^2(\R^n)$ can be written as
$$f= \sum_{I\in\DD(\R^n)}\langle f,\psi_I^\epsilon\rangle\,\psi_I^\epsilon.$$
To simplify notation and avoid using the $\epsilon$ index, we consider $2^n-1$ copies of $\DD(\R^n)$ and we denote by $\wt\DD(\R^n)$ their union. Then we can write 
$$f= \sum_{I\in\wt \DD(\R^n)}\langle f,\psi_I\rangle\,\psi_I,$$
with the sum converging in $L^2(\R^n)$.

In particular, we have 
\begin{equation}
\label{hti}
\wt h :=\chi_{B_n(0,1)} - \frac1{2^n}\,\chi_{B_n(0,2)} = \sum_{I\in\wt \DD(\R^n)} a_I\,\psi_I,
\end{equation}
where $B_n(0,r)$ stands for the ball centered at $0$ with radius $r$ in $\R^n$ and 
$$ a_I =  \Bigl\langle \chi_{B_n(0,1)} - \frac1{2^n}\,\chi_{B_n(0,2)},\,\psi_I\Bigr\rangle.$$
So we have
$$\frac1{r^n}\,\chi_{B_n(0,r)}(x) - \frac1{(2r)^n}\,\chi_{B_n(0,2r)}(x) = \sum_{I\in\wt \DD(\R^n)} a_I\,
\frac1{r^n} \psi_I\left(\frac xr\right).$$

Notice that we have been talking about wavelets in $\R^n$ although the ambient space of the measure $\mu$ and the function $h_r$ is $\R^d$, with $d\geq n$. We identify $\R^n$ with the ``horizontal'' subspace of $\R^d$ given by $\R^n\times\{0\}\times\ldots\times
\{0\}$ and
 we consider the following circular projection
$\Pi:\R^d\to\R^n$. For $x=(x_1,\ldots,x_d)\in\R^d$ we denote $x^H := (x_1,\ldots,x_n)$ and $x^V=(x_{n+1},\dots,x_d)$.
If $x^H\neq0$ we set
$$\Pi(x) = \frac{|x|}{|\xh|}\,\xh.$$
If $\xh=0$, we set $\Pi(x) = (|x|,0,\ldots,0)$, say. Observe that in any case $|x|=|\Pi(x)|$.

Notice also that 
$$h_r(x) = \frac1{r^n}\,\chi_{B_n(0,r)}(\Pi(x)) - \frac1{(2r)^n}\,\chi_{B_n(0,2r)}(\Pi(x)) = \sum_{I\in\wt \DD(\R^n)} a_I\,
\frac1{r^n} \psi_I\left(\frac{\Pi( x)}r\right).$$
Thus,
\begin{equation}
\label{hrmx}
h_r*\mu(x) = \sum_{I\in\wt \DD(\R^n)} a_I\,
\frac1{r^n} \psi_I\left(\frac{\Pi( \cdot)}r\right)*\mu(x).
\end{equation}

Observe that the functions $\psi_I$ are smooth, and so one can guess that the $\alpha$ coefficients of
\cite{Tolsa-lms} will be useful to estimate $\psi_I\left(\frac{\Pi( \cdot)}r\right)*\mu(x)$.
Concerning the coefficients $a_I$ we have:

\begin{lemma}
\label{boundaries}
For $I \in \wt \DD(\Rn)$, we have:
\begin{itemize}
\item[(a)] If $5I\cap \bigl(\partial B_n(0,1)\cup \partial B_n(0,2)\bigr)=\varnothing$, then $a_I=0$.
\medskip
\item[(b)] If $\ell(I)\gtrsim 1$, then $|a_I|\lesssim \ell(I)^{-1-n/2}$.
\medskip
\item[(c)] If $\ell(I) \lesssim 1$, then $|a_I|\lesssim \ell(I)^{n/2}$.
\end{itemize}
\end{lemma}

\begin{proof}
The first statement follows from the fact that the wavelets $\psi_I$ have zero mean in $\R^n$ and that 
$\wt h= \chi_{B_n(0,1)} - \frac1{2^n}\,\chi_{B_n(0,2)}$ is constant on $\supp\psi_I$ if 
$5I\cap \bigl(\partial B_n(0,1)\cup \partial B_n(0,2)\bigr)=\varnothing$.

The statement (c) is immediate:
$$|a_I| = \left|\int_{R^n} \wt h\,\psi_I\,dx
\right| \leq \|\psi_I\|_1\lesssim \ell(I)^{n/2}\,\|\psi_I\|_2 = \ell(I)^{n/2}.$$

Finally (b) follows from the smoothness of $\psi_I$ and the fact that $\wt h$ has zero mean. Indeed,
\begin{align*}
|a_I| = \left|\int_{B_n(0,2)} \wt h(x)\,(\psi_I(x)-\psi_I(0))\,dx
\right|
\leq 2\|\nabla\psi_I\|_\infty\,\int |\wt h|\,dx \lesssim \frac1{\ell(I)^{1+n/2}}.
\end{align*}
\end{proof}

\vvv

By estimating $\psi_I\left(\frac{\Pi( \cdot)}r\right)*\mu(x)$ in terms of the $\alpha(Q)$'s, using some
arguments in the spirit of the ones in \cite{Mas-Tolsa}, below we will prove the following.

\begin{theorem}\label{teoq1'}
Let $\mu$ be an $n$-AD-regular measure in $\R^d$.
If $\mu$ is uniformly $n$-rectifiable, then there exists a constant $c$ such that
\begin{equation}\label{eqsq11ns}
\int_0^R \int_{x\in B(x_0,R)}|h_r *\mu (x)|^2\,d\mu(x)\,\frac{dr}r \leq c\, R^n,
\end{equation}
for all $x_0\in\supp(\mu)$, $R>0$.
\end{theorem} 


\subsection{Preliminaries for the proof of Theorem \ref{teoq1'}}

 It is immediate to check that the estimate \rf{eqsq11ns} holds if and only if  for all $R\in\DD$
\begin{equation}\label{eqsqns}
\sum_{Q\in\DD:Q\subset R} \int_Q \int_{\ell(Q)}^{2\ell(Q)} |h_r*\mu(x)|^2\,\frac{dr}{\ell(Q)}\,d\mu(x)
\leq c\,\mu(R).
\end{equation}

Let $\de >0$  be some small constant to be fixed below. 
To estimate the preceding integral we can assume that $ \alpha (1000 Q) \leq \de^2$. Otherwise we have 
$$|h_r \ast \mu (x)| \lesssim 1 \leq \frac{\a (1000 Q)}{\de^2}$$ and, by Theorem \ref{teotol},
\begin{equation}\label{eqaasdq1}
\begin{split}
&\sum_{\substack{Q \in \DD(R) \\
\a(1000 Q) \geq\delta^2}} \frac{1}{\ell(Q)} \int_Q \int_{\ell(Q)}^{2 \ell(Q)}|h_r \ast \mu (x)|^2 dr d \mu (x)\\ 
&\lesssim \frac{1}{\de^4} \sum_{Q \in \DD(R)} \a(1000 Q)^2 \mu(Q) \lesssim \frac1{\delta^4} \mu (R).
\end{split}
\end{equation}

 Since the functions $h_r$ are even, we have
$$h_r \ast \mu (x)= \int h_r(y-x) d \mu (y).$$
Recalling \eqref{hrmx}, we get
$$h_r \ast \mu (x)=\frac1{r^n} \sum_{I\in\wt \DD(\R^n)} a_I\,
\int \psi_I\left(\frac{\Pi( y-x)}r\right)d\mu(y).$$
By Lemma \ref{boundaries}, $a_I=0$ whenever $5I\cap \bigl(\partial B_n(0,1)\cup \partial B_n(0,2)\bigr)=\varnothing$.
Therefore the preceding sum ranges over those $I$ such that $5I\cap \bigl(\partial B_n(0,1)\cup \partial B_n(0,2)\bigr)\neq \varnothing$ and the domain of integration of each $\psi_I\left(\frac{\Pi( \cdot)}r\right)$ is $\Pi^{-1}(r\cdot5I)$.

Notice that $5I$ stands for the cube from $\R^n$ concentric with $I$ with side length equal to $5\ell(I)$.
On the other hand, given a set $A\subset \R^d$, we write
$$r\cdot A= \{r\cdot x\in \R^d:\,x\in A\}.$$
So $r\cdot 5I = r\cdot (5I)$ is a cube in $\R^d$ with side length $5r\ell(I)$ which is not concentric with $I$ unless $I$ is centered at the origin.

We set
\begin{equation}
\label{hrab}
\begin{split}h_r \ast \mu (x)&=\frac1{r^n} \sum_{I \in\wt \DD(\R^n): \ell(I) \geq 1/100} a_I\,
\int \psi_I\left(\frac{\Pi( y-x)}r\right)d\mu(y)\\
&\quad \quad+\frac1{r^n} \sum_{I \in \wt \DD(\R^n): \ell(I) < 1/100} a_I\,
\int \psi_I\left(\frac{\Pi( y-x)}r\right)d\mu(y) \\
&=:F_r(x)+ G_r(x),
\end{split}
\end{equation}
so that
\begin{equation}
\label{i1i2}
\begin{split}
\sum_{Q \in \DD : Q \subset R} & \int_Q \int_{\ell(Q)}^{2\ell(Q)} |h_r*\mu(x)|^2\,\frac{dr}{\ell(Q)}\,d\mu(x)\\
&\lesssim 
\sum_{\substack{Q \in \DD : Q \subset R, \\
\a(1000 Q) \geq\delta^2}} \int_Q \int_{\ell(Q)}^{2\ell(Q)} |h_r \ast \mu (x)|^2 \,\frac{dr}{\ell(Q)}\,d\mu(x) \\
&\quad \quad\quad +
\sum_{\substack{Q \in \DD : Q \subset R, \\ \alpha (1000 Q) \leq \de^2}} \int_Q \int_{\ell(Q)}^{2\ell(Q)} |F_r(x)|^2\,\frac{dr}{\ell(Q)}\,d\mu(x) \\
&\quad \quad\quad +\sum_{\substack{Q \in \DD : Q \subset R, \\ \alpha (1000 Q) \leq \de^2}} \int_Q \int_{\ell(Q)}^{2\ell(Q)} |G_r(x)|^2\,\frac{dr}{\ell(Q)}\,d\mu(x)\\
&=:I_0 + I_1+I_2.
\end{split}
\end{equation}
As shown in \rf{eqaasdq1}, we have
$$I_0\lesssim \frac1{\delta^4}\,\mu(R).$$
Thus to prove Theorem \ref{teoq1'} it is enough to show that $I_1+I_2\leq c(\delta)\,\mu(R)$.



\subsection{Estimate of the term $I_1$ in \rf{i1i2}}

We first need to estimate $F_r(x)$. To this end, we take  $Q \in \DD$ and $r>0$  such that $x \in Q$ and $\ell(Q) \leq r < 2 \ell(Q)$. 
We also assume that $L_Q$ (the best approximating plane for $\a(Q)$) is parallel to $\R^n$. 

Let $I \in \wt \DD(\R^n)$ be such that $\ell(I) \geq 1/100$ and $5I\cap \bigl(\partial B_n(0,1)\cup \partial B_n(0,2)\bigr)\neq \varnothing$. Let $P:=P(I) \in \DD$ be some cube containing $Q$ such that
$\ell(P) \approx r \ell(I) \approx \ell(Q) \ell (I)$. Let also $\phi_P$ be a smooth bump function such that $\chi_{3P} \leq \phi_P \leq \chi_{B_P}$, $\| \nabla \phi_P\|_ \infty\leq1$,  and $\phi_P=1$ on $x+\Pi^{-1}(r\cdot5I)$.
Then
$$\int \psi_I\left(\frac{\Pi( y-x)}r\right)d\mu(y)=\int_{3P} \psi_I\left(\frac{\Pi( y-x)}r\right)d\mu(y)= \int \phi_P (y) \,\psi_I\left(\frac{\Pi( y-x)}r\right)d\mu(y).$$

\begin{lemma}\label{lemfrx0}
Let $I \in \wt \DD(\R^n)$ be such that $\ell(I) \geq 1/100$ and $5I\cap \bigl(\partial B_n(0,1)\cup \partial B_n(0,2)\bigr)\neq \varnothing$ and let $P=P(I)$ as above. We have
\begin{equation}
\label{integestim}
\left|\int \psi_I\left(\frac{\Pi( y-x)}r\right)d\mu(y)\right|\lesssim   \left(\frac  {\ell(Q)}{\ell(P)}\right)^{n/2} \left( \frac{\dist(x,L_Q)}{\ell(P)}+\sum _{S\in\DD:Q \subset S \subset P} \a(2S) \right) \ell(P)^n.
\end{equation}
\end{lemma}

\begin{proof}
Without loss of generality we assume that $x=0$.
Let $L_0$ be the plane parallel to $L_Q$ passing through $0$ (that is, $L_0=\R^n$) and denote by
$\Pi^\perp$ the orthogonal projection onto $L_0$.
  Then
\begin{equation}
\label{splitest1}
\begin{split}
\int \psi_I&\left(\frac{\Pi( y)}r\right)d\mu(y)=\int \phi_P (y) \psi_I\left(\frac{\Pi( y)}r\right)d\mu(y)\\
& =\int \phi_P (y)\ \left( \psi_I\left(\frac{\Pi( y)}r\right)-\psi_I\left(\frac{\Pi^\perp( y)}r\right)\right)d\mu(y)\\
&\quad \quad +\int \phi_P (y)\ \psi_I\left(\frac{\Pi^\perp( y)}r \right)d(\mu-c_P \ha^n_{|L_0})(y)\\
&\quad \quad \quad \quad+c_P \int_{L_0} \phi_P (y)\ \psi_I\left(\frac{\Pi^\perp( y)}r \right) d \ha^n (y) \\
&=:A_1+A_2+A_3.
\end{split}
\end{equation}
Since $\psi_I\left(\frac{\Pi^\perp( y)}r \right)= \psi_I \left(\frac yr\right)$ for $y \in L_0$, and $\phi_P=1$ on $r\cdot5I$, we get
\begin{equation}
\label{a3zero}
A_3=c_P\int_{5I} \psi_I \left(\frac{y}{r}\right)d y=0.
\end{equation}

We now proceed to estimate $A_2$:
\begin{equation}
\label{esta2}
\begin{split}
|A_2| &\leq \left|\int \phi_P (y)\ \psi_I\left(\frac{\Pi^\perp( y)}r \right)d(\mu-c_P \ha^n _{|L_P})(y) \right|
\\&\quad\quad+\left|c_P\int \phi_P (y)\ \psi_I\left(\frac{\Pi^\perp( y)}r \right)d(\ha^n_{|L_P}-\ha^n _{|L_0})(y) \right|\\
&\lesssim \left\| \nabla \left( \phi_P \ \psi_I\left(\frac{\Pi^\perp( \cdot)}r \right) \right)\right\|_\infty \, \bigl(\a(P) \, \ell(P)^{n+1}+ \dist_H (L_P \cap B_P, L_0 \cap B_P) \ell(P)^{n}\bigr),
\end{split}
\end{equation}
from the definition of the $\a$ numbers and the fact that $c_P \approx 1$. Using the gradient bounds for the functions $\phi_P$ and $\psi_I$, and the fact that $\ell(P) \approx r \ell(I) \approx \ell(Q) \ell (I),$ we get
\begin{equation}
\label{grada2}
\begin{split}
&\left\| \nabla \left( \phi_P \ \psi_I\left(\frac{\Pi^\perp( \cdot)}r \right) \right)\right\|_\infty \lesssim 
\|\nabla \phi_P\|_\infty \|\psi_I\|_\infty+\left\|\nabla \left(\psi_I \left(\frac{\Pi^\perp( \cdot)}r \right)\right)\right\|_\infty \\
& \quad \quad\lesssim \frac 1 {\ell(P)} \frac 1 {\ell(I)^{n/2}}+\frac 1 {\ell(I)^{n/2+1}}\frac 1r \approx \frac1 {\ell(Q)} \frac 1 {\ell(I)^{n/2+1}} \approx \frac1 {\ell(Q)} \left(\frac  {\ell(Q)}{\ell(P)}\right)^{n/2+1}.
\end{split}
\end{equation}
We also remark that in the previous estimate we used the fact that $\|\Pi^\perp\|_\infty \leq 1$, which does not hold for the spherical projection $\Pi$. 

Furthermore, by \cite[Lemma 5.2 and Remark 5.3]{Tolsa-lms},
\begin{equation}
\label{a2haus}
\begin{split}
\dist_H &(L_P \cap B_P, L_0 \cap B_P) \leq \dist_H (L_P \cap B_P, L_Q \cap B_P)+ \dist(0,L_Q)\\
&\leq \sum _{S\in\DD:Q \subset S \subset P} \a(S) \ell (P)+ \dist(0,L_Q).
\end{split}
\end{equation}
Therefore, by \eqref{esta2}, \eqref{grada2}, and \eqref{a2haus},
\begin{equation}
\label{a2}
|A_2| \lesssim \left(\frac  {\ell(Q)}{\ell(P)}\right)^{n/2}\, \ell(P)^n \left( \sum _{S\in\DD:Q \subset S \subset P} \a(S) + \frac{\dist(0,L_Q)}{\ell(P)}\right).
 \end{equation}
 
We now estimate the term $A_1$:
\begin{equation}
\label{a1}
\begin{split}
|A_1|& = \left|\int \phi_P (y)\ \left( \psi_I\left(\frac{\Pi( y)}r\right)-\psi_I\left(\frac{\Pi^\perp( y)}r\right)\right)d\mu(y)\right| \\
&\lesssim \frac{\| \nabla\psi_I\|_\infty }{r} \int_{B_P}|\Pi(y)-\Pi^\perp (y)|\, d \mu (y)\\
&\lesssim \frac1 {\ell(Q)} \left(\frac  {\ell(Q)}{\ell(P)}\right)^{n/2+1} \int_{B_P}|\Pi(y)-\Pi^\perp (y)| \, d \mu (y).
\end{split}
\end{equation}
It is easy to check that
\begin{equation}
\label{ppperp}
|\Pi(y)-\Pi^\perp (y)| \lesssim \dist(y, L_0).
\end{equation}
Furthermore, as in \eqref{a2haus}, for $y \in B_P$,
\begin{equation}
\label{a1haus}
\begin{split}
\dist(y, L_0) &\leq \dist(0,L_Q)+\dist(y,L_Q) \\
& \leq\dist(0,L_Q)+\dist(y,L_P)+\dist_H(L_P \cap 3B_P, L_Q \cap 3 B_P)\\
&\lesssim \dist(0,L_Q) +\dist(y,L_P)+ \sum _{S\in\DD:Q \subset S \subset P} \a(S) \ell (P).
\end{split}
\end{equation}
Therefore, by \eqref{a1}, \eqref{ppperp}, and \eqref{a1haus},
\begin{equation}
\label{a1fin}
\begin{split}
|A_1| \lesssim &\frac1 {\ell(Q)} \left(\frac  {\ell(Q)}{\ell(P)}\right)^{n/2+1} \Big(\dist(0,L_Q)\, \ell(P)^n\\
&\quad \quad+\int_{B_P} \dist(y,L_P)  \, d \mu(y)+ \ell (P)^{n+1}\sum _{S\in\DD:Q \subset S \subset P} \a(S) \Big)\\
&\lesssim \left(\frac  {\ell(Q)}{\ell(P)}\right)^{n/2} \left( \frac{\dist(0,L_Q)}{\ell(P)}+\sum _{S\in\DD:Q \subset S \subset P} \a(2S) \right) \ell(P)^n,
\end{split}
\end{equation}
where we used that, by \cite[Remark 3.3]{Tolsa-lms}, 
 $$\int_{B_P} \dist(y,L_P)  d \mu(y)\lesssim \a(2P) \ell(P)^{n+1}.$$

The lemma follows from the estimates \eqref{splitest1}, \eqref{a3zero}, \eqref{a2}, and \eqref{a1fin}.
\end{proof}

\begin{lemma}\label{lemfrx}
We have
\begin{equation}
\label{esta}
\begin{split}
|F_r(x)| \lesssim \frac{\dist(x,L_Q)}{\ell(Q)}+\sum _{S\in\DD: S \supset Q} \a(2S) \frac  {\ell(Q)}{\ell(S)}.
\end{split}
\end{equation}
\end{lemma}

\begin{proof}
Recalling that $P=P(I) \supset Q$, by \eqref{integestim},
\begin{equation*}
|F_r(x)| \lesssim \frac{1}{\ell(Q)^n} \sum_{\substack{I\in  \wt \DD(\R^n): \\  \ell(I) \geq 1/100}} \!\!|a_I| \left(\frac  {\ell(Q)}{\ell(P(I))}\right)^{n/2} \left( \frac{\dist(x,L_Q)}{\ell(P(I))}+\sum _{\substack{S\in\DD: \\ Q \subset S \subset P(I)}}\!\! \a(2S) \right) \ell(P(I))^n.
\end{equation*}
Using (b) from Lemma \ref{boundaries}, 
\begin{equation*}
\begin{split}
|F_r(x)| \lesssim & \sum_{\substack{I\in  \wt \DD(\R^n):  \ell(I) \geq 1/100, \\
P(I) \supset Q}} \frac  {\ell(Q)}{\ell(P(I))}\left( \frac{\dist(x,L_Q)}{\ell(P(I))}+\sum _{S\in\DD:Q \subset S \subset P(I)} \a(2S) \right)\\
& \lesssim \sum_{P \in \DD:\,P \supset Q} \frac  {\ell(Q)}{\ell(P)}\left( \frac{\dist(x,L_Q)}{\ell(P)}+\sum _{S\in\DD:Q \subset S \subset P} \a(2S) \right) \\
&=\sum_{P \in \DD:\,P \supset Q} \frac{\dist(x,L_Q)\, \ell(Q)}{\ell(P)^2}+\sum _{S\in\DD: S \supset Q} \a(2S) \sum_{P \in \DD:\,P \supset S} \frac  {\ell(Q)}{\ell(P)} \\
&\lesssim \frac{\dist(x,L_Q)}{\ell(Q)}+\sum _{S\in\DD: S \supset Q} \a(2S) \frac  {\ell(Q)}{\ell(S)}.
\end{split}
\end{equation*}
\end{proof}

\begin{lemma}\label{lemi1}
The term $I_1$ in \rf{i1i2} satisfies
$$I_1\lesssim\mu(R).$$
\end{lemma}

\begin{proof} By \eqref{esta},
\begin{equation}
\label{spliti1}
I_1 \lesssim  \sum_{Q \in \DD(R)} \int_Q \left(\frac{\dist(x,L_Q)}{\ell(Q)}+\sum _{S\in\DD: S \supset Q} \a(2S) \frac  {\ell(Q)}{\ell(S)} \right)^2 d\mu (x).
\end{equation}

By Cauchy-Schwartz,
\begin{equation}
\left( \sum _{S\in\DD: S \supset Q} \a(2S) \frac  {\ell(Q)}{\ell(S)} \right)^2 \leq  \sum _{S\in\DD: S \supset Q} \a(2S)^2 \frac{\ell(Q)}{\ell(S)}\; \cdot\!\! \sum _{S\in\DD: S \supset Q} \frac{\ell(Q)}{\ell(S)}. 
\end{equation}
Since $\sum _{S\in\DD: S \supset Q} \frac{\ell(Q)}{\ell(S)} \lesssim 1$,
\begin{equation}
\begin{split}
\label{i1}
I_1 &\lesssim  \sum_{Q \in \DD(R)} \int_Q \left(\frac{\dist(x,L_Q)}{\ell(Q)}\right)^2
\,d\mu(x)+\sum_{Q \in \DD(R)}  \sum _{S\in\DD: S \supset Q} \a(2S)^2 \frac{\ell(Q)}{\ell(S)} \mu(Q)\\
&=:S_1+S_2.
\end{split}
\end{equation}
By \cite[Lemmas 5.2 and Lemma 5.4]{Tolsa-lms} and Theorem \ref{teotol}, we obtain $S_1 \lesssim \mu(R)$. We now deal with the term $S_2$:
\begin{equation}
\label{s2split}
\begin{split}
S_2&=\sum_{Q \in \DD(R)}  \sum _{S\in\DD: Q \subset S \subset R} \a(2S)^2 \frac{\ell(Q)}{\ell(S)} \mu(Q)+\sum_{Q \in \DD(R)}  \sum _{S\in\DD: S\supsetneq R} \a(2S)^2 \frac{\ell(Q)}{\ell(S)} \mu(Q)
\\
&=:S_{21}+S_{22}.
\end{split}
\end{equation}
Using just that $\a(2S) \lesssim 1$,
\begin{equation}
\label{s22}
S_{22} \lesssim \sum_{Q \in \DD(R)}  \sum _{S\in\DD: S\supset R } \frac{\ell(Q)}{\ell(S)} \mu(Q) \lesssim \sum_{Q \in \DD(R)}  \mu(Q) \frac{\ell(Q)}{\ell(R)} \lesssim \mu (R).
\end{equation}

Finally, using Fubini and Theorem \ref{teotol},
\begin{equation}
\label{s21}
S_{21} \leq \sum_{S \in \DD(R)} \a(2S)^2 \sum _{Q\in\DD: Q \subset S}  \frac{\ell(Q)}{\ell(S)} \mu(Q) \lesssim \sum_{S \in \DD(R)} \a(2S)^2 \mu(S) \lesssim \mu(R).
\end{equation}
By \eqref{i1}, \eqref{s2split}, \eqref{s22}, and \eqref{s21} we obtain $I_1 \lesssim \mu(R)$.
\end{proof}

\vspace*{0.2in}


\subsection{Estimate of the term $I_2$ in \rf{i1i2}}

It remains to show that $I_2 \lesssim \mu(R)$. Recall that the cubes in the sum corresponding to $I_2$
in \rf{i1i2} satisfy $\a(1000\,Q) \leq \delta^2$. 

We need now to estimate $G_r(x)$ (see \eqref{hrab}) for $x\in Q$ and $\ell(Q)\leq r<2\ell(Q)$. 
Recall that
\begin{equation}
\label{b}
G_r(x)=\frac1{r^n} \sum_{I \in \wt \DD(\R^n): \ell(I) < 1/100} a_I\,
\int \psi_I\left(\frac{\Pi( y-x)}r\right)d\mu(y).
\end{equation}
The arguments will be more involved than the ones we used for $F_r(x)$.

To estimate $G_r(x)$ we now introduce a \textit{stopping time condition} for $P \in \DD$: $P$ belongs to $\cG_0$ if
\begin{enumerate}
\item $P \subset 1000\,	Q$, and
\item $\sum_{S \in \DD: P \subset S \subset 1000Q} \a(100S) \leq \delta$.
\end{enumerate} 
The maximal cubes in $\DD\setminus\cG_0$ may vary significantly in size, even if they are neighbors, and this would cause problems. For this reason we use a quite standard \textit{smoothing} procedure. We define
\begin{equation}
\label{ell}
\ell(y):=\inf_{P\in \cG_0}(\ell(P)+\dist(y,P)),\quad y \in \R^d,
\end{equation}
and
\begin{equation}
\label{dfun}
d(z):=\inf_{y\in \Pi^{-1}(z)}\ell(y), \quad z \in \R^n .
\end{equation}
\begin{lemma}
\label{lip}
The function $\ell(\cdot)$ is $1$-Lipschitz, and the function $d(\cdot)$ is $3$-Lipschitz.
\end{lemma}

\begin{proof} For simplicity we assume that $x=0$.
The function $\ell(\cdot)$ is $1$-Lipschitz, as the infimum of the family of $1$-Lipschitz functions $\{\ell(P)+\dist(\cdot,P)\}_{P \in \GG_0}$.

Let us turn our attention to $d(\cdot)$. 
Let $z,z' \in \R^n$ and $\ve>0$. Let $y \in \Pi^{-1}(z)$ such that $\ell(y) \leq d(z)+ \ve$. Consider the points
$$y_0 =  \frac{|z'|}{|y|}\,y
\quad\text{ and }\quad 
z_0 = \frac{|z'|}{|z|}\,z.$$ 
Notice that $\Pi(y_0) = z_0$. Let $L_{y_0}$ be the $n$-plane parallel to $\Rn$ which contains $y_0$ and consider the point $\{y'\} =\Pi^{-1}(z') \cap L_{y_0}$.
That is, $y'$ is the point which fulfils the following properties:
$$|y'|=|z'|,\quad y'^H = \frac{|y'^H|}{|z'|}\,z',\quad y'^V=y_0^V.$$
Observe that 
$$|y_0|=|z_0|=|z'|=|y'|.$$
Since $y_0^V=y'^V$, this implies that $|y_0^H|=|y'^H|$.
 Furthermore,
$$|y_0-y'|=|y_0^H-y'^H|=\left|\frac{|y_0^H|}{|y_0|}z_0- \frac{|y'^H|}{|y'|}z'\right|=\frac{|y'^H|}{|y'|}|z_0-z'|\leq |z_0-z'|.$$
Moreover, 
$$|y-y_0| = |z-z_0| = \biggl|z- \frac{|z'|}{|z|}\,z\biggr| = \bigl||z|-|z'|\bigr| \leq |z-z'|.$$
Hence,
\begin{equation*}
\begin{split}
\label{3lip}
|y-y'| &\leq |y-y_0|+|y_0-y'| \leq |z-z'|+|z_0-z'| \\
&\leq |z-z'|+|z_0-z|+|z-z'| \leq 3|z-z'|.
\end{split}
\end{equation*}
Then, using that $\ell$ is $1$-Lipschitz and \eqref{3lip},
\begin{equation*}
d(z') \leq \ell(y')\leq |y-y'|+\ell(y) \leq 3|z-z'|+d(z)+\ve.
\end{equation*}
Since $\ve>0$ was arbitrary we deduce that $d(z') \leq 3|z-z'|+d(z)$. In the same way one gets that $d(z) \leq 3|z-z'|+d(z')$.
\end{proof}

For $\delta$ small enough, the condition $\a(1000\,Q) \leq \delta^2$ guarantees that any cube $P \subset 1000\,Q$ such that $\ell(P)=\ell(Q)$ belongs to $\cG_0$, in particular $\ell(y) \leq \ell(Q)$ for all $y \in 1000 \, Q$. Furthermore since $\cG_0 \neq \varnothing$, we deduce that $\ell(y), \, d(z)< \infty$ for all $y \in \R^d, \, z \in \R^n$.

Now we consider the family $\cF$ of cubes $I \in \DD(\R^n)$ such that
\begin{equation}
\label{notes2}
r \,\diam(I) \leq  \frac{1}{5000} \inf_{z \in r\cdot I} d(z).
\end{equation}
Let $\cF_0 \subset \cF$ be the subfamily of $\cF$ consisting of cubes with maximal length. In particular the cubes in $\cF_0$ are pairwise disjoint.
 Moreover it is easy to check that if $I,J \in \cF_0$ and 
\begin{equation}
\label{notes3}
20 \, I \cap 20 \,J \neq \varnothing,
\end{equation}
then $\ell(I) \approx \ell(J)$.

We denote by $\cG (x,r)$ the family of cubes $I \in \wt \DD(\Rn)$ which satisfy
\begin{itemize}
\item $\ell(I) \leq \frac 1{100}$,
\medskip
\item $5I \cap (\partial B_n(0,1) \cup \partial B_n (0,2)) \neq \varnothing$,
\medskip
\item $\bigl(x+\Pi^{-1}(r\cdot5I)\bigr)\cap\supp(\mu)\neq \varnothing$, and
\medskip
\item $I$ is not contained in any cube from $\cF_0$.
\end{itemize}
We denote by $\cT (x,r)$ the family of cubes $I \in \wt \DD(\Rn)$ which satisfy
\begin{itemize}
\item $\ell(I) \leq \frac 1{100}$,
\medskip
\item $5I \cap (\partial B_n(0,1) \cup \partial B_n (0,2)) \neq \varnothing$,
\medskip
\item $\bigl(x+\Pi^{-1}(r\cdot5I)\bigr)\cap\supp(\mu)\neq \varnothing$, and
\medskip
\item $I \in \cF_0$.
\end{itemize}
Now we write
\begin{equation}
\label{gt}
\begin{split}
G_r(x)&=\frac1{r^n} \sum_{I \in \cG (x,r)} a_I\,
\int \psi_I\left(\frac{\Pi( y-x)}r\right)d\mu(y)\\
&\quad \quad+\frac1{r^n} \sum_{I \in \cT (x,r)} \sum_{J\in \wt\DD(\Rn):\, J \subset I} a_J\,
\int \psi_J\left(\frac{\Pi(y-x)}r\right)d\mu(y)\\
&=:G_{r,1}(x)+G_{r,2}(x),
\end{split}
\end{equation}
so that
\begin{equation}
\label{i1i2in}
\begin{split}
I_2&\leq \sum_{\substack{Q \in \DD : Q \subset R, \\ \alpha (1000 Q) \leq \de^2}}
 \int_Q \int_{\ell(Q)}^{2\ell(Q)} |G_{r,1}(x)|^2\,\frac{dr}{\ell(Q)}\,d\mu(x)\\
&\quad\quad\quad\quad+\sum_{\substack{Q \in \DD : Q \subset R, \\ \alpha (1000 Q) \leq \de^2}} \int_Q \int_{\ell(Q)}^{2\ell(Q)} |G_{r,2}(x)|^2\,\frac{dr}{\ell(Q)}\,d\mu(x)\\
&=: I_{21}+I_{22}.
\end{split}
\end{equation}

First we will deal with the term $G_{r,1}(x)$. 
To this end we need several auxiliary lemmas. 



\begin{lemma}
\label{key}
If $I \in \cG (x,r)$, then there exists $P:=P(I) \in \DD$ with $\ell(P) \approx r \ell(I)$ such that
$$\supp (\mu) \cap \bigl(x+\Pi^{-1}(r\cdot 5I)\bigr) \subset 3P.$$
\end{lemma}

\begin{proof}
Notice that, by definition, $\supp (\mu) \cap (x+ \Pi^{-1}(r\cdot 5I))\neq\varnothing$.
Observe also that the conclusion of the lemma holds if $\ell(r\cdot5I)\approx\ell(Q)$ because $\alpha(1000Q)
\leq\delta^2$.

So assume that $\ell(r\cdot5I)\ll\ell(Q)$ and consider $z \in r\cdot5I$.
Since $I \in \cG(x,r)$, $I \notin \cF_0$, and $d$ is $3$-Lipschitz, we have
$$d(z) \leq c_2\,r \ell(I),$$
for some absolute constant $c_2$.
Take $y\in x+\Pi^{-1}(r\cdot5I)$ such that
$$\ell(y)\leq 2c_2\,r\ell(I).$$
Let $\ve= c_2\,r\ell(I)$. By definition, there exists some cube $P'\in \GG_0$ such that
$$\ell(P')+\dist(y,P') \leq \ell(y)+\ve \leq 3c_2\,r\,\ell(I).$$

Let $A>10$ be some big constant to be fixed below.
Suppose that there are two cubes $P_0,P_1 \in \DD$ which satisfy the following properties
\begin{itemize}
\item[(i)]  $r \ell (I) \leq \ell(P_0)= \ell(P_1) \leq 10\, r \ell(I)$,
\item[(ii)] $\dist(P_0, P_1) \geq A \ell(P_0),$
\item[(iii)] $P_i \cap \bigl(x+\Pi^{-1}(r\cdot5I)\bigr) \neq \varnothing$ for $i=1,2$.
\end{itemize}
Suppose that $\dist(P_0,P')\geq \dist(P_1,P')$. Then from (ii) we infer that
$$\dist(P_0,P')\gtrsim A\,\ell(P_0).$$

Let $P'' \in \DD$ such that $P_0 \cup P' \subset 3P''$ with minimal side length, so that $\ell(P'')\approx  \ell(P_0)+ \ell(P')+\dist(P_0,P').$ Since $\a(1000 Q) \leq \de^2$ and $\ell(P_0),\ell(P_1),\ell(P')\ll\ell(Q)$, it follows
easily that we must also have $\ell(P'')\ll\ell(Q)$. It is not difficult to check that either 
$$\beta_1(P'') \gg \delta,\quad \mbox{or} \quad \mang (L_{P''}, L_Q) \gg \delta.$$
In either case one has
$$\sum_{S \in \DD:\, P'' \subset S \subset Q}  \a(S) \gg \delta.$$
We deduce that
$$\sum_{S \in \DD:\, P' \subset S \subset Q}  \a(100S) \gg \delta,$$
because $P'\subset 3P''$.
This contradicts  the fact that $P'\in \GG_0$.

We have shown that a pair of cubes $P_0,P_1$ such as the ones above does not exist.
Thus, if $P_0\in\DD$ satisfies
$$r \ell (I) \leq \ell(P_0)\leq 10\, r \ell(I),$$
and
$$P_0 \cap \bigl(x+\Pi^{-1}(r\cdot5I)\bigr) \neq \varnothing,$$
then any other cube $P_1$ for which these properties also hold must be contained
in the ball $B(x_{P_0},c_3\,A\,\ell(P_0))$, where $x_{P_0}$ stands for the center of $P_0$ and $c_3$ is some absolute constant.
Hence letting $P=P(I)$ be some suitable ancestor of $P_0$, the lemma follows.
\end{proof}

\vspace{2mm}

\begin{lemma}\label{lemgr100}
Let $I \in \cG (x,r)$ and let $P=P(I) \in \DD$ be the cube from Lemma \ref{key}, so that $\supp(\mu)\cap(x+\Pi^{-1}(r\cdot5I)) \subset 3P.$  We have
\begin{equation}
\label{gest}
\left|\int \psi_I\left(\frac{\Pi( y-x)}r\right)d \mu(y) \right| \lesssim \left( \frac{\ell(Q)}{\ell(P)} \right)^{n/2} \left(\sum_{S \in \DD: P \subset S \subset Q}\a(S) +\frac{\dist(x,L_Q)}{\ell(Q)}  \right) \ell(P)^n.
\end{equation}
\end{lemma}

\begin{proof} Without loss of generality we assume that $x=0$ and as before we let $L_0=\R^n$ be the $n$-plane parallel to $L_Q$ containing $0$. Let also $y_P \in B_P \cap \supp (\mu)$ be such that $\dist(y_P,L_P) \lesssim \a(P) \ell (P)$. The existence of such point follows from \cite[Remark 3.3]{Tolsa-lms} and Chebychev's inequality. We also denote by $\wt L_P$ the $n$-plane parallel to $L_0$ which contains $y_P$. We set $\sigma_P = c_P \ha^{n}_{|L_P}$ and $\wt \sigma_P=c_P\ha^{n}_{|\wt L_P}$. Let $\phi_P$ be a smooth function such that $\chi_{B_P} \leq \phi_P \leq \chi_{3 B_P}$ and $\|\nabla \phi_P\|_\infty \lesssim \ell(P)^{-1}$. 
Since $\alpha(P)$ is assumed to be very small, we have
$\Pi^{-1}(r\cdot5I)\cap \wt L_P\subset B_P$.
Then we write
\begin{equation}
\label{splitestg}
\begin{split}
\int \psi_I&\left(\frac{\Pi( y)}r\right)d\mu(y)=\int \phi_P (y) \,\psi_I\left(\frac{\Pi( y)}r\right)d\mu(y)\\
&=\int \phi_P (y) \psi_I\left(\frac{\Pi( y)}r\right)(d\mu(y)-d \sigma_P(y))\\
&\quad\quad+\int \phi_P (y) \psi_I\left(\frac{\Pi( y)}r\right)(d \sigma_P(y)-d \wt\sigma_P(y))+\int \psi_I\left(\frac{\Pi( y)}r\right)d \wt\sigma_P(y)\\
&=:A_1+A_2+A_3.
\end{split}
\end{equation}

Now we turn our attention to $A_1$:
\begin{equation}
\begin{split}
\label{g1}
|A_1|&=\left|\int \phi_P (y) \psi_I\left(\frac{\Pi( y)}r\right)(d\mu(y)-d \sigma_P(y)) \right| \\
&\leq \left\| \nabla \left(\phi_P \psi_I\left(\frac{\Pi( \cdot)}r\right)\right)\right\|_\infty \a(P) \ell(P)^{n+1}\\
&\lesssim \left(\frac{1}{\ell(P)} \frac{1}{\ell(I)^{n/2}}+\frac{1}{\ell(I)^{n/2+1}} \frac 1r\right)\a(P) \ell(P)^{n+1}\\
&\approx  \left( \frac{\ell(Q)}{\ell(P)}\right)^{n/2}\a(P) \ell(P)^{n},
\end{split}
\end{equation}
where we used that  $\ell(P) \approx \ell(I) \ell(Q)$ and that $\|\nabla \Pi\|_\infty \lesssim 1$ on $B_P$ since $B_P$ lies far from the subspace ${\Pi^\perp}^{-1}(\{0\})$.

We will now estimate the term $A_2$. We have
$$|A_2|=\left| \int \phi_P (y) \psi_I\left(\frac{\Pi( y)}r\right)(d \sigma_P(y)-d \wt\sigma_P(y)) \right|.$$
As in  \cite[Lemma 5.2]{Tolsa-lms},
$$\mang(L_P,\wt L_P)=\mang(L_P,L_Q)\lesssim \sum_{S \in \DD: P \subset S \subset Q} \a(S).$$
Therefore,
$$\dist_H(\wt L_P \cap B_P, L_P \cap B_P) \lesssim \sum_{S \in \DD: P \subset S \subset Q}\a(S) \ell(P),$$
and, as in \eqref{g1},
\begin{equation}
\label{g2}
\begin{split}
|A_2| &\lesssim \left\| \nabla \left(\phi_P \psi_I\left(\frac{\Pi( \cdot)}r\right)\right)\right\|_\infty\, \ell(P)^n \,\dist_H(\wt L_P \cap B_P, L_P \cap B_P)\\
&\lesssim \left( \frac{\ell(Q)}{\ell(P)} \right)^{n/2} \sum_{S \in \DD: P \subset S \subset Q}\a(S) \ell(P)^n.
\end{split}
\end{equation}

We now consider $A_3$.
Let $B$ be a ball centered in $L_0$ such that $\supp \,\psi_I \left( \frac{\cdot}{r}\right) \subset B$ and $\diam (B) \lesssim \ell(P)$. For some constant $c_*$, with $0\leq c_*\lesssim1$, to be fixed below, we write
\begin{equation}
\label{g3}
\begin{split}
\left|\int \psi_I\left(\frac{\Pi( y)}r\right)d \wt\sigma_P(y) \right|&=\left| c_P \int \psi_I\left(\frac{ y}r\right)d ( \Pi_\sharp \ha^n_{|\wt L_P})(y) \right|\\
&\leq \left| c_P \int \psi_I\left(\frac{ y}r\right)d (\Pi_\sharp \ha^n_{|\wt L_P})(y)
- c_*\,c_P \int \psi_I\left(\frac{ y}r\right) d \ha^n_{|L_0} (y) \right| \\
&\quad\quad\quad+\left|c_*\,c_P \int \psi_I\left(\frac{ y}r\right) d \ha^n_{|L_0} (y) \right|\\
&\lesssim  \frac{ \|\nabla \psi_I\|_\infty}{\ell(Q)} \,\dist_B(\Pi_\sharp \ha^n_{|\wt L_P}, c_* \ha^n_{|L_0}),
\end{split}
\end{equation}
where in the last inequality we took into account that $c_*\,c_P\lesssim1$ and that $\int_{\Rn}\psi_I \left( \frac{ y}r \right) dy=0$.

Notice that the map $\Pi_{|\wt L_P \ra L_0}$ need not be affine and so the term
$\dist_B(\Pi_\sharp \ha^n_{|\wt L_P}, c_* \ha^n_{|L_0})$ requires some careful analysis. Anyway, we claim that, for some appropriate constant $c_*\lesssim1$,
\begin{equation}\label{eqdaqp9}
\dist_B(\Pi_\sharp \ha^n_{|\wt L_P}, c_* \ha^n_{|L_0}) \lesssim \Biggl(\sum_{S \in \DD:P \subset S \subset Q} \a(S)+\frac{\dist(0,L_Q)}{\ell(Q)}\Biggr)\, \ell(P)^{n+1},
\end{equation}
which implies that
$$
|A_3|\lesssim \left( \frac{\ell(Q)}{\ell(P)} \right)^{n/2} \Biggl(\sum_{S \in \DD:P \subset S \subset Q} \a(S)+\frac{\dist(0,L_Q)}{\ell(Q)}\Biggr)\, \ell(P)^{n}.
$$
Notice that the lemma is an immediate
consequence of the estimates we have for $A_1$, $A_2$ and $A_3$.

To conclude, it remains to prove the claim \rf{eqdaqp9}. This task requires some preliminary calculations and we defer it to Lemma \ref{disttilds}.
\end{proof}

\vv
Our next objective consists in comparing the measures $\Pi_\sharp \ha^n_{|\wt L_P}$ and $\ha^n_{|L_0}$ from
the preceding lemma.
To this end, we consider the map $\wt \Pi:= \Pi_{|\wt L_P \ra L_0}$. Abusing notation, identifying both $\wt L_P$ and $L_0$ with $\Rn$, we also denote by $\wt \Pi$ the corresponding mapping in $\Rn$, that is $\wt \Pi: \Rn \ra \Rn$. Then, writing $h=y_P^V$, for $y=(y_1,\ldots,y_n,h)$ we have
$$\wt \Pi_i (y)=y_i \sqrt{\frac{y_1^2+\dots+y_n^2+|h|^2}{y_1^2+\dots+y_n^2}}=y_i \sqrt{1+\frac{|h|^2}{y_1^2+\dots+y_n^2}},$$
for $i=1,\dots,n$. Hence, for $i,j=1,\dots,n$,
$$\partial_j \wt \Pi_i=\delta_{ij}\frac{|y|}{|y^H|}- \frac{|h|^2\,y_i\,y_j}{|y||y_H|^3}=\frac{|y|}{|y_H|} \left(\delta_{ij}-|h|^2\frac{y_i y_j}{|y|^2|y_H|^2} \right),$$
where $\delta_{ij}$ denotes Kronecker's delta. For $y \in P$,
\begin{equation*}
\begin{split}
|\partial_j\wt\Pi_i(y)-\partial_j \wt \Pi_i (y_P)|\leq \left|\frac{|y|}{|y^H|}-\frac{|y_P|}{|y_P^H|}\right|+|h|^2\left| \frac{y_i\,y_j}{|y|\,|y^H|}-\frac{{y_P}_i\,{y_P}_j}{|y_P|\,|{y_P}^H|} \right|.
\end{split}
\end{equation*}
Moreover,
\begin{equation*}
\begin{split}
\left|\frac{|y|}{|y^H|}-\frac{|y_P|}{|y_P^H|}\right|&\approx \left|\frac{|y|^2}{|y^H|^2}-\frac{|y_P|^2}{|y_P^H|^2}\right|=\left| \frac{|y|^2|y_P^H|^2-|y_P|^2|y^H|^2}{|y^H|^2\,|y_P^H|^2} \right|\\
&=\left| \frac{(|y^H|^2+|h|^2)|y_P^H|^2-(|y_P^H|^2+|h|^2)|y^H|^2}{|y^H|^2\,|y_P^H|^2} \right|\\
&=\left| \frac{|h|^2(|y_P^H|^2-|y^H|^2)}{|y^H|^2\,|y_P^H|^2} \right| \approx \frac{|h|^2\,|y|\,||y_P^H|-|y^H||}{|y|^4} \\
&\lesssim \frac{|h|^2\, \ell(P)}{r^3},
\end{split}
\end{equation*}
and in a similar manner we get 
$$|h|^2\left| \frac{y_i\,y_j}{|y|\,|y^H|}-\frac{{y_P}_i\,{y_P}_j}{|y_P|\,|{y_P}^H|} \right|\lesssim \frac{|h|^2\, \ell(P)}{r^3}.$$
Hence
\begin{equation}
\label{der}
|\partial_j\wt\Pi_i(y)-\partial_j \wt \Pi_i (y_P)| \lesssim \frac{|h|^2\, \ell(P)}{r^3}.
\end{equation}
Now we write
\begin{equation}
\label{jacpre}
\begin{split}
|J \wt\Pi (y)-J \wt\Pi (y_P)|&=\left|\sum_{\sigma} \sgn(\sigma) \prod_{j=1}^n \partial_j \wt \Pi_{\sigma(j)}(y)-\sum_{\sigma} \sgn(\sigma) \prod_{j=1}^n \partial_j \wt \Pi_{\sigma(j)}(y_P)\right|\\
&\leq c(n) \sup_{i,j}\bigl|\partial_j \wt \Pi_i(y)-\partial_j \wt\Pi_i(y_P)\bigr| \,\bigl(\sup_{i,j}|\partial_j \wt \Pi_i(y)|^{n-1}+\sup_{i,j}|\partial_j \wt \Pi_i(y_P)|^{n-1}\bigr) \\
&\lesssim \sup_{i,j}\bigl|\partial_j \wt \Pi_i(y)-\partial_j \wt\Pi_i(y_P)\bigr|,
\end{split}
\end{equation}
where the sum is computed over all permutations of $\{1,\dots,n\}$ and $\sgn (\sigma)$ denotes the signature of the permutation $\sigma$. Moreover, in the last inequality we used again that  $\|\nabla \Pi\|_\infty \lesssim 1$ on $B_P$ since $B_P$ lies far from the subspace ${\Pi^\perp}^{-1}(\{0\})$.

Therefore, by \eqref{jacpre} and \eqref{der},
\begin{equation}
\label{jac}
|J \wt\Pi (y)-J \wt\Pi (y_P)|\lesssim \frac{|h|^2\, \ell(P)}{r^3}\qquad\mbox{for $y \in P$}.
\end{equation}

\begin{lemma} 
\label{disttilds}
Let $B$ be a ball centered in $\Pi (P)$ with $\diam(B) \lesssim \ell(P)$. Then
$$\dist_B(\Pi_\sharp \ha^n_{|\wt L_P}, c_* \ha^n_{|L_0}) \lesssim \left(\sum_{S \in \DD:P \subset S \subset Q} \a(S)+\frac{\dist(0,L_Q)}{\ell(Q)}\right)\, \ell(P)^{n+1},$$
where $c_*=(J \wt \Pi(y_P))^{-1}$.
\end{lemma}

\begin{proof} Let $f$ be $1$-Lipschitz with $\supp f \subset B$. Then, recalling that $\wt \sigma_P=c_P\ha^{n}_{|\wt L_P}$,
\begin{equation*}
\begin{split}
\left| \int f \,d (\Pi_\sharp \ha^n_{|\wt L_P})-c_*\int f\, d \ha^n_{|L_0} \right| &\approx \left|\frac{1}{c_*} \int f(\Pi (y))d \ha^n_{|\wt L_P}-\int f(y) d \ha^n_{|L_0}\right|\\
&= \left|\frac{1}{c_*}  \int_{\Rn} f(\wt \Pi (y)) dy-\int_{\Rn} f(y)dy\right|\\
&=\left| \int_{\Rn} f(\wt \Pi (y))\, J \wt \Pi(y_P) dy-\int_{\Rn} f(\wt \Pi(y)) \,J\wt \Pi(y)dy\right|,
\end{split}
\end{equation*}
where we changed variables in the last line. Now notice that $\supp f \circ \wt\Pi \subset B'$, where $B'$ is a ball concentric with $B$ such that $\diam(B) \lesssim \ell(P)$. In addition, since $\supp f \subset B$ and $\| \nabla f\|_\infty \leq 1$ we also get  $\|f\|_\infty \lesssim \ell(P)$. Hence, by \eqref{jac},
\begin{equation*}
\begin{split}
\left| \int f \,d (\Pi_\sharp \ha^n_{|\wt L_P})-c_*\int f\, d \ha^n_{|L_0} \right|& \lesssim \int_{\Rn} |f(\wt \Pi (y))||J \wt \Pi (y_P)- J \wt \Pi(y)|dy \\
& \lesssim \frac{|h|^2\, \ell(P)}{r^3} \int_{B'} \ell(P) dy 
\lesssim \frac{|h|\, \ell(P)^{n+1}}{r}.
\end{split}
\end{equation*}
Moreover, by \cite[Remark 5.3]{Tolsa-lms} and the choice of $y_P$,
$$|h|=\dist(y_P,L_0) \leq \dist(y_P,L_Q)+ \dist(L_0,L_Q)\lesssim \sum_{S \in \DD: P \subset S \subset Q} \a(S) \ell(S)+\dist(0,L_Q).$$
Hence
\begin{equation*}
\left| \int f\, d (\Pi_\sharp \ha^n_{|\wt L_P})-c_*\int f \,d \ha^n_{|L_0} \right| \lesssim \frac{\sum_{S \in \DD: P \subset S \subset Q} \a(S) \ell(Q)+\dist(0,L_Q)}{\ell(Q)}\, \ell(P)^{n+1},
\end{equation*}
and the lemma follows.
\end{proof}


We denote $$\wt \cG (x,r):= \{P(I)\}_{I \in \cG(x,r)}.$$
We need the following auxiliary result.

\begin{lemma}
\label{lemma4}
For every $a\geq1$ and every $S\in\DD$,
$$\sum_{P  \in \wt \cG(x,r):\, P \subset a\,S} \mu(P) \lesssim \mu (S),$$
with the implicit constant depending on $a$.
\end{lemma}
\begin{proof} We assume $x=0$ for simplicity.
Notice that for every $P  \in \wt \cG(0,r)$ such that $ P \subset a\,S$ there exists some $I \in \cG(0,r)$ such that $r \ell(I) \approx \ell(P)$ and $r\cdot I \subset a' \, \Pi (B_S)$ where $a'$ only depends on $a$. Therefore
\begin{equation*}
\begin{split}
&\sum_{P  \in \wt \cG(0,r);\, P \subset a\,S} \mu(P) \lesssim \sum\{\ell(r\cdot I)^n: I \in \cG (0,r); \, 
r\cdot I \subset a' \Pi (B_S)\}\\
&\lesssim \sum \{\ell(r \cdot I)^n: I \in \wt\DD(\Rn);\,r\cdot I \subset a' \Pi (B_S);  r\cdot5I \cap (\partial B_n(0,r) \cup \partial B_n(0,2r)) 
\neq \varnothing \} \\
&\\
&\leq c(a)\,\ell(S)^n \approx c(a)\,\mu(S).
\end{split}
\end{equation*}
\end{proof}


We can now estimate the term $G_{r,1}(x)$ in \eqref{gt}. 

\begin{lemma}\label{lemfaqw1}
We have
\begin{equation}
\label{gest2}
|G_{r,1}(x)| \lesssim  \sum_{P  \in \wt \cG(x,r)} \left( \a(aP)+\frac{d(x,L_Q)}{\ell(Q)} \right)\frac{\mu(P)}{\mu(Q)},
\end{equation}
for some absolute constant $a\geq1$.
\end{lemma}

\begin{proof}
Using \eqref{gest} and (c) from Lemma \ref{boundaries}, 
\begin{equation}
\begin{split}
|G_{r,1}(0)|&=\left| \frac1{r^n} \sum_{I \in \cG (0,r)} a_I\,
\int \psi_I\left(\frac{\Pi( y)}r\right)d\mu(y) \right| \\
&\lesssim \frac 1{\ell(Q)^n} \sum_{I \in \cG (0,r)} \left(\sum_{S \in \DD: P(I) \subset S \subset Q}\a(S) +\frac{\dist(0,L_Q)}{\ell(Q)}  \right) \ell(P(I))^n.
\end{split}
\end{equation}
 Notice that by the definition of $\cG(0,r)$, for every $I \in \cG(0,r)$
$$\# \{ P \in \DD: P=P(I) \} \lesssim 1.$$
Then
\begin{equation}
\label{4}
|G_{r,1}(0)| \lesssim \sum_{P\in \wt \cG(0,r)} \sum_{S \in \DD: P \subset S \subset Q}\a(S)\, \frac{\ell(P)^n}{\ell(Q)^n}+\sum_{P \in \wt\cG(0,r)}\frac{\dist(0,L_Q)}{\ell(Q)}\frac{\ell(P)^n}{\ell(Q)^n}.
\end{equation}
If $S\in \DD$ is such that $P \subset S \subset Q$, then there exists $\wt S \in \wt \cG (0,r)$, with $\ell(\wt S) \approx \ell(S)$, such that $S \subset a \wt S$ for some $a \geq 1$. In fact, since $P \in \wt \GG (0,r)$ we can find $I' \in \GG(0,r)$ with $\ell(r\cdot I')\approx\ell(S)$ such that $\Pi (S) \cap r\cdot I' \neq \varnothing$. Therefore we can take $\wt S:=P(I')$. 

Hence for $P \in \wt \cG (0,r)$,
\begin{equation}
\sum_{S \in \DD: P \subset S \subset Q}\a(S) \lesssim \sum_{\substack{ S \in \wt\cG(0,r):\\ P \subset a\, S \subset a\,Q}}\a(a\,S).
\end{equation}
Thus, using also Lemma \ref{lemma4},
\begin{equation}
\begin{split}
\label{4a}
\sum_{P \in \wt \cG(0,r)} \sum_{S \in \DD: P \subset S \subset Q}\a(S) \,\frac{\ell(P)^n}{\ell(Q)^n} &\lesssim \sum_{P  \in \wt \cG(0,r)} \sum_{\substack{ S \in \wt\cG(0,r):\\ P \subset a\, S \subset a\,Q}}\a(a\,S) \frac{\ell(P)^n}{\ell(Q)^n}\\
&\approx\sum_{\substack{S  \in \wt \cG(0,r):\,S \subset a\,Q}}\a(a\,S) \sum_{P  \in \wt \cG(0,r):\, P \subset a\,S} \frac{\mu(P)}{\mu (Q)}
\\
&\lesssim\sum_{\substack{S  \in \wt \cG(0,r):\,S \subset a\,Q}}\a(a\,S) \,\frac{\mu(S)}{\mu (Q)}.
\end{split}
\end{equation}
Together with \rf{4}, this yields \rf{gest2}.
\end{proof}



Now we will deal with the term $I_{21}$ in \rf{i1i2in}. 

\begin{lemma}\label{lemi21}
We have
$$I_{21}\lesssim\mu(R).$$
\end{lemma}

\begin{proof}0
By Lemmas \ref{lemma4} and \ref{lemfaqw1},
and Cauchy-Schwarz,
\begin{equation*}
\begin{split}
|G_{r,1}(x)|^2& \lesssim \left(\sum_{P  \in \wt \cG(x,r)} \left( \a(a P)+\frac{d(x,L_Q)}{\ell(Q)} \right)^2\frac{\mu(P)}{\mu(Q)} \right)\left( \sum_{P  \in \wt \cG(x,r)} \frac{\mu(P)}{\mu(Q)} \right)\\
&\lesssim \sum_{P  \in \wt \cG(x,r)} \left( \a(a P)+\frac{d(x,L_Q)}{\ell(Q)} \right)^2\frac{\mu(P)}{\mu(Q)}.
\end{split}
\end{equation*}
Then
\begin{equation*}
\begin{split}
I_{21}&\lesssim \sum_{Q\in \DD(R)} \frac1{\ell(Q)^{n+1}} \int_Q \int_{\ell(Q)}^{2 \ell(Q)} \sum_{P  \in \wt \cG(x,r)} \left( \a(a P)^2+\frac{d(x,L_Q)^2}{\ell(Q)^2} \right)\mu(P)\, dr d\mu(x)\\
&\lesssim \sum_{Q\in \DD(R)} \frac1{\ell(Q)^{n+1}} \int_Q \int_{\ell(Q)}^{2 \ell(Q)} \!\!\!\sum_{ \substack{P  \subset a''Q:\\
cB_P \cap  (\partial B(x,r) \cup \partial B(x,2r)) \neq \varnothing}} \!\!\left( \a(a P)^2+\frac{d(x,L_Q)^2}{\ell(Q)^2} \right)\mu(P)\, dr d\mu(x).
\end{split}
\end{equation*}
By Fubini,
\begin{equation*}
\begin{split}
\int_{\ell(Q)}^{2 \ell(Q)} &\sum_{ \substack{P  \subset a''Q:\\
  cB_P \cap  \partial B(x,r) \neq \varnothing}} \left( \a(a P)^2+\frac{d(x,L_Q)^2}{\ell(Q)^2} \right)\mu(P)\, dr\\
&= \sum_{ \substack{P  \subset a''Q}} \left( \a(a P)^2+\frac{d(x,L_Q)^2}{\ell(Q)^2} \right)\mu(P) \int_{\{r:\, cB_P \cap  \partial B(x,r) \neq \varnothing\}}dr\\
&\lesssim \sum_{ \substack{P  \subset a''Q}} \left( \a(a P)^2+\frac{d(x,L_Q)^2}{\ell(Q)^2} \right)\mu(P) \ell(P),
\end{split}
\end{equation*}
where we used the fact that if $r>0$ is such that $ cB_P \cap  \partial B(x,r) \neq \varnothing$ then $$|x_P|-c\, \ell(P)\leq r \leq |x_P|+c\, \ell(P),$$ where $x_P$ is the center of $B_P$. Therefore, 
\begin{equation}
\begin{split}
I_{21}&\lesssim \sum_{Q\in \DD(R)}\frac 1{\ell(Q)^n} \int_Q \sum_{ \substack{P  \subset a''Q}} \left( \a(a P)^2+\frac{d(x,L_Q)^2}{\ell(Q)^2} \right) \frac{\ell(P)}{\ell(Q)} \mu(P) d \mu (x)\\
&\lesssim \sum_{Q\in \DD(R)} \sum_{ \substack{P  \subset a''Q}}\a(a P)^2  \frac{\ell(P)}{\ell(Q)} \,\mu (P)\\
&\quad\quad \quad +\sum_{Q\in \DD(R)} \frac 1{\ell(Q)^n} \int_{Q}\frac{d(x,L_Q)^2}{\ell(Q)^2} d \mu(x) \sum_{P\in \DD: P \subset a''Q} \frac{\ell(P)}{\ell(Q)} \,\mu(P)\\
&\lesssim \sum_{P\in \DD: P \subset a''R} \a(a P)^2 \mu (P) \sum_{Q\in \DD: a''Q \supset P}\frac{\ell(P)}{\ell(Q)}+ \sum_{Q\in \DD(R)}  \int_{Q}\frac{d(x,L_Q)^2}{\ell(Q)^2} d \mu(x) \\
&\lesssim \mu (R).
\end{split} 
\end{equation}
\end{proof}

Finally we turn our attention to $I_{22}$. Recall 
that
$$I_{22} = \sum_{\substack{Q \in \DD : Q \subset R, \\ \alpha (1000 Q) \leq \de^2}} \int_Q \int_{\ell(Q)}^{2\ell(Q)} |G_{r,2}(x)|^2\,\frac{dr}{\ell(Q)}\,d\mu(x).$$
For $x\in \supp (\mu)$ and $r>0$ set 
$$f_{x,r}(y)=\sum_{I \in \cT (x,r)} \sum_{J\in \wt\DD(\Rn):\, J \subset I} a_J\,
\psi_J\left(\frac{\Pi( y-x)}r\right),$$
so that
$$G_{r,2}(x) = \frac1{r^n}\int f_{x,r}(y)\,d\mu(y).$$

\begin{lemma}
\label{lemma5}
The functions $f_{x,r}$ satisfy
\begin{itemize}
\item $\supp f_{x,r} \subset \bigcup_{I \in \cT (x,r)} 3 P(\hat{I})$, where $\hat{I}$ is the father of $I$,
\medskip
\item $\|f_{x,r}\|_\infty \lesssim 1$.
\end{itemize}
\end{lemma}

\begin{proof}
We assume again that $x=0$. Notice that $\supp f_{x,r} \subset \Pi^{-1}(r\cdot5 I) \cap \supp(\mu)$ and since $I \in \cF_0$, we have $\hat{I} \in \cG(x,r)$. Therefore by Lemma \ref{key}, 
$\Pi^{-1}(r\cdot5 I) \cap \supp(\mu) \subset 3 P(\hat{I})$.

We will now show that $\|f_{x,r}\|_\infty \lesssim 1$. Recalling \eqref{notes3} if $I, J \in \cF_0$ and $20 I \cap 20 J \neq \varnothing$, then $\ell(I) \approx \ell(J)$. If $I \in \cF_0  \setminus \cT(x,r)$ or $I \subset J$ for some $J \in \cF_0  \setminus \cT(x,r)$, then by Lemma \ref{boundaries} $a_I=0$. Therefore,
$$f_{x,r}(y)= \sum_{I \in \cF_0} \sum_{J\in \wt\DD(\Rn):\, J \subset I}a_J \psi_J \left( \frac{\Pi(y)}{r}\right).$$

We now consider the function 
$$\wt f(z)=\sum_{I \in \cF_0} \sum_{J \subset I}a_J \psi_J(z).$$
The second assertion in the lemma follows after checking that $\|\wt f\|_\infty \lesssim 1$. To this end, recall that by \eqref{hti}, for any $k \in \Z$, we have
$\wt h= \sum_{I \in \wt\DD (\Rn)} a_I \psi_I$. We can also write
\begin{equation}
\label{5notes}
\wt h (z)= \sum_{I \in \wt \DD_k(\Rn)} \sum_{J \subset I} a_J \psi_J (z)+ \sum_{I \in \DD_k (\Rn)} \beta_I \phi_I(z),
\end{equation}
where $\beta_I= \langle \wt h, \phi_I \rangle$ and the functions $\phi_I$ satisfy
\begin{itemize}
\item $\supp \phi_I \subset 7I$,
\medskip
\item $\|\phi_I\|_\infty \lesssim  \frac 1{\ell(I)^{n/2}}$,
\medskip
\item $\|\nabla \phi_I\|_\infty \lesssim  \frac 1{\ell(I)^{n/2+1}}$,
\medskip
\item $\|\phi_I\|_2 =1$.
\end{itemize}
See \cite[Theorem 7.9]{mal}. We note that $\supp \psi_I \subset 5I$ and $\supp \phi_I \subset 7I$ since we are taking Daubechies wavelets with $3$ vanishing moments, see \cite[p.\ 250]{mal}.  

Now let $z \in I_0$ for some $I_0 \in\cF_0$ with $\ell(I_0)=2^{-k}$.  Notice that
\begin{equation}
\begin{split}
\label{bibound}
\|\beta_I \phi_I\|_\infty &\lesssim |\beta_I| \ell(I)^{-n/2} \leq \int |\wt h(y) \, \phi_I(y)|dy \,\,\ell(I)^{-n/2} \\
\\& \lesssim \| \phi_I\|_1 \,\ell(I)^{-n/2} \lesssim \ell(I)^{n/2} \| \phi_I \|_2 \ell(I)^{-n/2}=1.
\end{split}
\end{equation}
By the finite superposition of $\supp \phi_I$ for $ I \in \DD_k,$ \eqref{bibound} implies that
$$\left| \sum_{I \in \DD_k (\Rn)} \beta_I \phi_I(z)\right|\lesssim 1.$$
Therefore by \eqref{5notes} we deduce that
\begin{equation}
\label{6}
\left| \sum_{I \in \wt \DD_k(\Rn)} \sum_{J \subset I} a_J \psi_J (z) \right|\lesssim 1.
\end{equation}

We will now prove that
$$
\left| \wt f(z)-\sum_{I \in \wt \DD_k(\Rn)} \sum_{J \subset I} a_J \psi_J (z) \right| \lesssim 1.
$$
Together with \eqref{6}, this shows that $|\wt f(z)| \lesssim 1$ and proves the lemma. We have
\begin{equation}
\begin{split}
\left| \wt f(z)-\sum_{I \in \wt \DD_k(\Rn)} \sum_{J \subset I} a_J \psi_J (z) \right|&=\left|\sum_{\substack{I \in\cF_0:\\ 5I \cap I_0 \neq \varnothing}} \sum_{J \subset I} a_J \psi_J (z)-\sum_{\substack{I \in \DD_k (\Rn):\\ 5I \cap I_0 \neq \varnothing}} \sum_{J \subset I} a_J \psi_J (z)\right|\\
& \lesssim \sum_{J \in A_1 \bigtriangleup A_2} |a_J \psi_J (z)|,
\end{split}
\end{equation}
where
$$A_1=\{J \in \wt\DD(\Rn): J \subset I, \mbox{ for some } I \in \cF_0\mbox{ such that } 
5I \cap I_0 \neq \varnothing\}$$
and
$$A_2=\{J \in \wt\DD(\Rn): J \subset I, \mbox{ for some } I \in \DD_k(\Rn)\mbox{ such that } 5I \cap I_0 \neq \varnothing\}.$$
It follows as in \eqref{bibound} that $\|a_J \psi_J\|_\infty \lesssim 1$. Therefore,
$$
\left| \wt f(z)-\sum_{I \in \wt \DD_k(\Rn)} \sum_{J \subset I} a_J \psi_J (z) \right| \lesssim \# A\lesssim1.
$$
This follows from the fact that if $I \in\cF_0$ such that $5I \cap I_0 \neq \varnothing$ then $\ell(I)  \approx \ell(I_0)$.
\end{proof}

\begin{lemma}\label{lemi22}
We have
$$I_{22}\lesssim\mu(R).$$
\end{lemma}

\begin{proof}
Lemma \ref{lemma5} implies that 
$$|G_{r,2}(x)| \lesssim  \frac1{\ell(Q)^n} \int |f_{x,r}(y)| d\mu (y) \lesssim\frac1{\ell(Q)^n} \sum_{I \in \cT (x,r)} \mu (P (\hat{I})).$$
As noted earlier, for $I\in\cT(x,r)$,  the parent of $I$, denoted by $\hat{I}$, belongs to $\cG (x,r)$.
Observe also that
$$r\,\diam(I)\leq \frac1{5000}\inf_{z\in r\cdot I}d(z),$$
because $\cT(x,r)\in\FF_0$. So every $z'\in r\cdot I\subset r\cdot\hat I$ satisfies $d(z')\geq5000 \,r\,\diam(I)
=2500\,r\,\diam(\hat I)$. This implies that $d(z)\gtrsim r\,\ell(I)$ for all $z\in r\cdot\hat I$, because $d(\cdot)$ is $3$-Lipschitz.
As a consequence, by the definition of $d(\cdot)$, there exists some $y\in P(\hat I)$
such that $\ell(y)\gtrsim r\,\ell(I)\approx\ell(P(\hat I))$.
Then it follows easily that there exists some descendant $U$ of $P(\hat{I})$ with $\ell(U) \approx \ell(P(\hat{I}))$ such that 
\begin{equation*}
\sum_{S \in \DD: U \subset S \subset 1000 Q} \a(100 S) \geq \delta.
\end{equation*}
This clearly implies that either
$$\sum_{S \in \DD: 	P(\hat{I}) \subset S \subset 1000 Q} \a(100 S) \geq \frac\delta2,$$
or 
$$\sum_{S \in \DD: U \subset S \subset P(\hat{I})} \a(100 S) \geq \frac\delta2.$$
Since $\ell(U) \approx \ell(P(\hat{I}))$, from the second condition one infers that $\a(100 P(\hat{I})) \geq c \delta.$ Hence in either case, for some small constant $c>0$,
\begin{equation}
\sum_{S \in \DD: P \subset S \subset 1000Q} \a (100 S) \geq c \delta.
\end{equation}
Therefore,
\begin{equation*}
\begin{split}
|G_{r,2}(x)| &\lesssim_\delta \frac1{\ell(Q)^n} \sum_{I \in \cT (x,r)} \mu (P (\hat{I}))\sum_{S \in \DD: P \subset S \subset 1000Q} \a (100 S)\\
&\lesssim \frac{1}{\ell(Q)^n}\sum_{P \in  \wt\cG (x,r)} \mu (P)\sum_{S \in \DD: P \subset S \subset 1000Q} \a (100 S).
\end{split}
\end{equation*}
Notice that
$$\sum_{P \in  \wt\cG (x,r)} \frac{\ell(P)^n}{\ell(Q)^n}\sum_{S \in \DD: P \subset S \subset 1000Q} \a (100 S)$$
is smaller, modulo the constants $1000$ and $100$, than the right side in \eqref{4}. Therefore by the same arguments we used for $I_{21}$ we get 
$I_{22} \lesssim \mu(R).$
\end{proof}

From Lemmas \ref{lemi21} and \ref{lemi22} we deduce that 
$I_2\lesssim \mu(R).$
Together with Lemma \ref{lemi1} this completes the proof of Theorem \ref{teoq1'}.

\vv


\section{Proof of Proposition \ref{prop1}}\label{sec_prop}

We will only prove the equivalence (a)$\Leftrightarrow$(c), as (a)$\Leftrightarrow$(b) is very similar.

By Theorem \ref{teoq1}, 
it is clear that uniform $n$-rectifiability implies the boundedness of the square function in (c) for any positive integer $k$. As for the converse, next we show that Lemma \ref{lemcompact} holds with $\wt\Delta_{\mu,\vphi}$ replaced by $\wt\Delta^k_{\mu,\vphi}$. Recall 
$$
\wt\Delta^k_{\mu,\vphi}(x,t) = \int \partial^k_\vphi(x-y,t)\ d\mu(y),\;\mbox{ where } \;\partial^k_\vphi(x,t) = t^k\partial_t^k\vphi_t(x).
$$

\begin{lemma}\label{cor_lemcompact}
Let $k\geq 1$ and let $\mu$ be an $n$-AD-regular measure such that $0\in\supp(\mu)$. For all
$\ve>0$ there exists  $\delta>0$ such that
if 
$$\int_{\delta}^{\delta^{-1}}\!\!
 \int_{x\in \bar B(0,\delta^{-1})} |\wt \Delta^k_{\mu,\vphi} (x,r)|\,d\mu(x)\,dr \leq \delta,$$
then
$$\dist_{B(0,1)}(\mu,\wt\UU(\vphi,c_0)) <\ve.$$
\end{lemma}

\begin{proof}
Suppose that there exists an $\ve>0$, and for each $m\geq 1$ there exists an $n$-AD-regular
 measure $\mu_m$ such that $0\in\supp(\mu_m)$,
\begin{equation}\label{cor_eqass32}
\int_{1/m}^{m}
 \int_{x\in \bar B(0,m)} |\wt \Delta^k_{\mu_m,\vphi} (x,r)|\,d\mu_m(x)\,dr \leq \frac1m,
\end{equation}
and
\begin{equation}\label{cor_equu12}
\dist_{B(0,1)}(\mu_m,\wt\UU(\vphi,c_0)) \geq\ve.
\end{equation}

By \eqref{ad} we can replace $\{\mu_m\}$ by a subsequence converging  weak * (i.e.\ when tested against 
compactly supported continuous functions) to a measure $\mu$ and it is easy to check that $0 \in \supp(\mu)$ and that $\mu$ is also $n$-dimensional AD-regular with constant $c_0$. We claim that 
\begin{equation*}
\int_{0}^\infty\!
 \int_{x\in\R^d} |\wt \Delta^k_{\mu,\vphi} (x,r)|\,d\mu(x)\,dr=0.
 \end{equation*}\
The proof of this statement is elementary and is almost the same as the analogous one in Lemma \ref{lemcompact}.
We leave the details for the reader.


Our next objective consists in showing that $\mu\in \wt\UU(\vphi,c_0)$. To this end, denote by $G$ the subset of those points $x\in\supp(\mu)$ such that
$$\int_{0}^\infty\!
 |\wt \Delta^k_{\mu,\vphi} (x,r)|\,dr=0.$$
It is clear that $G$ has full $\mu$-measure. For $x\in G$ and $r>0$, consider the function
$f_x(r) = \vphi_r*\mu(x)$.
Notice that $f_x:(0,+\infty)\to\R$ is $C^\infty$ and satisfies
$$\partial_r^k f_x(r) = (\partial_r^k \vphi_r) *\mu(x) = r^{-k}\,\wt \Delta^k_{\mu,\vphi} (x,r)
=0$$
for a.e.\ $r>0$. Thus $f_x$ is a polynomial on $r$ of degree at most $k-1$, whose coefficients may depend on $x$.
However, since $\mu$ is $n$-AD-regular, it follows easily that there exists some constant $c$ such that
$$|f_x(r)| = |\vphi_r*\mu(x)|\leq c\quad\mbox{for all $r>0$.}$$
Thus $f_x$ must be constant on $r$. So for all $x\in G$ and $0<R_1\leq R_2$,
$$\vphi_{R_1} *\mu(x) = \vphi_{R_2}*\mu(x).$$
This is the same estimate we obtained in \rf{eqig2} in Lemma \ref{lemcompact}. So proceeding exactly
in the same way as there we deduce then that
$$\vphi_{R_1} *\mu(x) = \vphi_{R_2}*\mu(y) \quad\mbox{for all $x,y\in\supp\mu$ and all $0<R_1\leq R_2$.}$$
That is,
 $\mu\in
 \wt\UU(\vphi,c_0)$.
 However, by condition \rf{cor_equu12}, letting $m\to\infty$, we have 
$$\dist_{B(0,1)}(\mu,\wt\UU(\vphi,c_0)) \geq\ve,$$
because $\dist_{B(0,1)}(\cdot,\wt\UU(\vphi,c_0)) $ is continuous under the weak * topology.
So $\mu\not\in\wt\UU(\vphi,c_0)$, which is a contradiction.
\end{proof}

Applying the previous lemma and arguing in the same way as in Section \ref{sec:bddur} one proves
the implication (c)$\Rightarrow$(a) of Proposition \ref{prop1}.


\end{document}